\documentclass[10pt]{article}

\usepackage[utf8]{inputenc}
\usepackage[T1]{fontenc}
\usepackage{amssymb}
\usepackage{amsmath}
\usepackage{graphicx}
\usepackage{epsfig}
\usepackage{arydshln}
\usepackage[dvipsnames]{xcolor} 
\usepackage{pgfplots}

\usepackage[left=24mm, textwidth=160mm]{geometry}

\newtheorem{theorem}{Theorem}

\newtheorem{corollary}[theorem]{Corollary}

\newtheorem{lemma}[theorem]{Lemma}

\newtheorem{remark}[theorem]{Remark}

\begin{document}


\title{Two Multi--Draw Coupon Collector models with different retention rules}

\author{Aristides V. Doumas\thanks{Department of Mathematics,
        School of Applied Mathematical and Physical Sciences,
        National Technical University of Athens, Zografou Campus,
        15780 Athens, Greece (adou@math.ntua.gr).}
        \thanks{Archimedes/Athena Research Center, Greece.}\,
        and
        S. Spektor\thanks{School of Data, Computing and Mathematics, Canisius University, 2001 Main Street, Buffalo, NY 14208-1098, USA (spektors@canisius.edu)\\[6pt]} }

\date{}

\maketitle

\vspace{-6mm}

\begin{abstract}
In this paper we study two variants of the generalized coupon collector's problem, where our collector receives at each run $d$ distinct coupons and keeps all the new observed coupons (Problem I), while he chooses the least--collected coupon at each run (Problem II). In both cases we derive explicit formulae for the average of the random variable denoting the number of trials for a complete set of $N$ different types of coupons, which are uniformly distributed. In both cases we present the asymptotic expansion up to the fourth term including the corresponding error term. Then, for both problems we derive the full asymptotic expansion as $N\rightarrow \infty$. We further obtain the leading-order behaviour of the variance, showing that in both problems $\mathrm{Var}\sim \frac{\pi^2}{6}\frac{N^2}{d^2}$, and we establish a rate of convergence to the limiting law. Our analysis is based on the Nørlund–Rice integral method applied to an alternating binomial sum and classical tools from asymptotic analysis. The leading asymptotic term for Problem II was obtained by W. Xu and A. K. Tang [\textit{J. Appl. Probab.} \textbf{48} (2011), 1081--1094]. Finally, for both problems, we derive the limiting distribution under the appropriate normalization. As expected, the limit is standard Gumbel; however, the normalization differs between Problems I and II. As an application, we show that Problem~I describes exactly the sequencing-coverage process in combinatorial motif-based DNA data storage, and our expansions yield closed-form coverage estimates for that setting.
\end{abstract}


\smallskip

\vspace{-1mm}

\textbf{MSC 2020 Mathematics Classification.}  60C05, 60F05, 41A60, 05A16.

\vspace{-1mm}

\section{Introduction and motivation}
Coupon collector problem (CCP) is probably the most popular \textit{Urn problem} due to its mathematical elegance, as well as its applications in several areas of science, from computer science  (search algorithms) and biology, to physics, linguistics, ecology, earth and planetary sciences, economics and finance, as well as, demography, and the social sciences, see, e.g., \cite{BH}, \cite{DH}. The original problem dates back to De Moivre’s treatise \textit{De Mensura Sortis} (1712) and Laplace's \emph{Th\'eorie analytique des probabilit\'es} (1812). For the history of (CCP) see, e.g., \cite{DIXIE} and the references therein. In the classic version of the problem coupons are drawn independently with replacement from a population of $N$ different types which are uniformly distributed. Naturally, the main object of study was the waiting time until all types of coupons have been collected (at least once). In the uniform case the expectation of the waiting time is  $NH_N$ (where $H_N$ is the $N$-th Harmonic number, see, e.g., W. Feller’s classical work \cite{F}), while the problem of collecting $m$ complete sets of coupons is known as the double Dixie cup problem, and was studied by Newman and Shepp, \cite{NS}. A few years later, Erdős and Rényi derived the corresponding limiting distribution revealing the extreme-value nature of the final stage of the collection process, see, \cite{ER}. The problem was related to the Dixie Cup Company, since in the 1930’s the company started a unique program where children collected Dixie lids to receive "Premiums" starting with illustrations of Dixie Circus characters. \footnote{This initiative was part of a broader strategy to promote the brand and engage children, which later expanded to include Hollywood stars and sports personalities during World War II. The program not only entertained children but also helped the company maintain its market presence during challenging times. Historical information on the Dixie Cup Company can be found in the following link: \textit{https://sites.lafayette.edu/dixiecollection/company-history.}}
In its general form (CCP) refers to a population whose members are of $N$ different \emph{types}, and for $j=1,2,\cdots,N,$ we denote by $p_j$ the probability that a randomly chosen coupon is of type $j$, where $p_j > 0$ and $\sum_{j=1}^{N}p_{j}=1$, and the $p_j$'s are not necessarily uniformly distributed.  The members of the population are (always) sampled independently \textit{with replacement} and their types are recorded. A substantial number of papers have investigated variants of the classical coupon collector problem regarding mainly, the distribution of the  coupon probabilities, or even changes in the sampling mechanism. These include generalized Zipf and logarithmic Zipf laws, as well as mixtures and interlacings of different coupon families. The main objects of study were not only the leading asymptotic term of the expectation, but also refined asymptotic expansions, rising moments, the leading behavior of the variance, and the limiting distribution under the appropriate normalization (in most cases, though not always, is Gumbel) see, e.g. (\cite{DP} -- \cite{DIP}) and the references therein.\\
Nevertheless, the (CCP) remains a remarkably fertile source of new variants and continues to provide fruitful directions for further research. For example, we refer the reader to \cite{SIB} and \cite{BR} for a very entertaining variation of the (CCP). In this paper we consider two variants of the uniform (CCP) by changing the sampling mechanism itself. In particular, our collector receives at each run $d$ distinct coupons and keeps all the new observed coupons \textbf{(Problem I)}, while he chooses the least--collected coupon at each run \textbf{(Problem II)}; (here and in what follows $d$ is a positive integer). Thus, the parameter $d$ introduces a choice or acceleration mechanism into the coupon collecting process. We will refer to this variant as the \textbf{(d- CCP)}. Notice, that the case where, $d=1$ is the classic version of the problem.\\
The outline of the paper follows. Section 2 is devoted to \textbf{(Problem I)}. We denote by $T_{N,d}$ the number of trials until all $N$ different types of coupons, which are uniformly distributed, are collected (at least once). Then, we present a closed form for the average $\mathbb {E}\left[T_{N,d}\right]$ in Theorem~\ref{theorem1}, the first four terms plus the error in its asymptotic expansion in Theorem~\ref{result}, and the full asymptotic expansion of $\mathbb {E}\left[T_{N,d}\right]$ in Theorem~\ref{thm:fullI}. The main tools are the Nørlund–Rice integral method applied to an alternating binomial sum and classical asymptotic tools, such as the celebrated Euler–Maclaurin summation formula, as well as complex-analytic techniques. In Section 3 we study the average of the random variable $T_{N,1}$ (i.e., the number of trials until all $N$ different types of the uniformly distributed coupons are collected (at least once), under the least--collected retention rule). The results are presented in Theorems~\ref{theorem12}, \ref{result2}, and~\ref{thm:fullII} respectively (however, the  techniques are different). Finally, in Section 4 we prove that both the studied random variables appropriately normalized converge in distribution to a standard Gumbel random variable. Although the limiting distribution is the same in both cases, the normalizations \textit{differ} reflecting the way the $d$-draw mechanism affects the final phase of the collection process.  The last missing coupons still determine the limiting law, but the scale on which they are collected is altered by the $d$-draw rule. In Section 5 we obtain the leading-order behaviour of the variance in both problems, and in Section 6 we describe an application to sequencing coverage in combinatorial motif-based DNA data storage. Notice that a closely related problem to our (Problem II) was considered in \cite{WT}. In particular, although the selection rule is the same, the stopping rule is different: the process stops when $m$ complete sets of all coupon types have been collected. 
Their method relies mainly on comparison arguments producing asymptotically matching upper and lower bounds. For  $m=1$ the result of \cite{WT} gives just the leading term in the asymptotic expansion of the mean of $T_{N,1}$. Some proofs are gathered in the Appendix (Section 8).

\begingroup
For the reader's convenience, Table~\ref{tab:compare} summarises the two models and our main
results side by side. The central structural finding is that the two retention rules produce
\emph{identical} leading and logarithmic behaviour, differing only in the linear coefficient,
through the constant $C_d$ of~(\ref{0100}).
\begin{table}[h!]
\centering
\renewcommand{\arraystretch}{1.6}
\begin{tabular}{l|c|c}
 & \textbf{Problem I} (keep all new) & \textbf{Problem II} (keep least-collected)\\\hline
random variable & $T_{N,d}$ & $T_{N,1}$\\
exact mean & Theorem~\ref{theorem1} & Theorem~\ref{theorem12}\\
$\mathbb{E}[\,\cdot\,]$, leading & $\dfrac{N}{d}\log N+\dfrac{\gamma}{d}N$ & $\dfrac{N}{d}\log N+\Big(\dfrac{\gamma}{d}+C_d\Big)N$\\
$\log N$ coefficients & \multicolumn{2}{c}{identical (Remark~\ref{rem:d3})}\\
full expansion & Theorem~\ref{thm:fullI} & Theorem~\ref{thm:fullII}\\
variance, leading & $\dfrac{\pi^2}{6}\dfrac{N^2}{d^2}$ & $\dfrac{\pi^2}{6}\dfrac{N^2}{d^2}$\\
normalised limit & $G$ (Gumbel), scale $N/d$ & $G$ (Gumbel), scale $N/d$\\
centring & $\dfrac{N}{d}\log N$ & $\dfrac{N}{d}\log N+NC_d$
\end{tabular}
\caption{The two models and the main results of the paper at a glance.}
\label{tab:compare}
\end{table}
\endgroup

\section{Problem  I - $\mathbb {E}[T_{N,d}]$}

\begin{theorem}  \label{theorem1}
Consider the  $d$ coupon collector's problem (d--CCP), where in each run we keep \textbf{all} the new observed coupons. Let $T_{N,d}$ be the number of trials until all $N$ different types of coupons, which are uniformly distributed, are collected (at least once). Then,
\begin{equation}
\mathbb {E}\left[T_{N,d}\right]=\binom{N}{d} \sum_{k=1}^{N}\left(-1\right)^{k+1}\binom{N}{k}\Bigg(\binom{N}{d}-\binom{N-k}{d}\Bigg)^{-1}.
\label{1}
\end{equation}
\end{theorem}
\textbf{Proof}. It is convenient to introduce the events $A_{j}^{k}$, $1 \leq j \leq N$, that the type $j$ is not detected until trial $k$ (included). Then,
\begin{align}
P\left\{T_{N,d} \geq k\right\}&=
P\left(A_{1}^{k}\cup A_{2}^{k}\cup \cdots \cup A_{N}^{k}\right)\qquad k=1,2,\ldots \nonumber\\
&=\sum_{\substack{J \subset \{1,\ldots,N\} \\ J \neq \varnothing}}
(-1)^{|J|-1}P\Bigg(\bigcap_{j \in J} A^{k-1}_j \Bigg) \nonumber\\
&=\sum_{m=1}^{N}\left(-1\right)^{m-1}\binom{N}{m}\Bigg(\frac{\binom{N-m}{d}}{\binom{N}{d}}\Bigg)^{k-1}, \label{dix}
\end{align}
where, in the second equality we have used the inclusion–exclusion principle, while the corresponding sum extends over the $2^{N}-1$ non empty subsets $J$ of $\left\{1,...,N\right\}$, and, $|J|$ denotes the cardinality of $J$. Since $T_{N,d}$ is a non-negative integer-valued random variable we have
\begin{align*}
\mathbb {E}\left[P\left\{T_{N,d} \geq k\right\}\right]=\sum_{k=1}^{\infty}\sum_{m=1}^{N}\left(-1\right)^{m-1}\binom{N}{m}\Bigg(\frac{\binom{N-m}{d}}{\binom{N}{d}}\Bigg)^{k-1}.
\end{align*}
Now the result follows immediately by interchanging the order of summation and adding the resulting geometric series.  $\hfill \blacksquare$
\begin{remark}
The result of Theorem~\ref{theorem1} may also be obtained by formulating the problem as a finite Markov chain and applying standard potential theory for hitting times (see, e.g., \cite{N}).\end{remark}
The problem of estimating \textit{asymptotically} high order differences of the type of the RHS of (\ref{1}) is delicate. The basic approach to the asymptotic analysis of sums of the form (\ref{1}) is well known as the technique of \textit{Rice or (Nørlund–Rice) integrals}. It has found numerous applications in the analysis of algorithms. In particular, in the study of digital trees, digital search trees, and quadtrees, among others. We refer the interested reader to the excellent work of P. Flajolet and R. Sedgewick \cite{FS} and the references therein.

\begin{lemma}  \label{firstlemma}
Let $\phi(s)$ be analytic in a domain that contains the half--line $[n_{0},\infty)$. Then, the differences of the sequence $\left\{\phi(k)\right\}$ admit the integral representation
\begin{equation}
\sum_{k=n_{0}}^{N}\binom{N}{k}\left(-1\right)^{k+1}\phi(k)=\frac{\left(-1\right)^{N+1}}{2i \pi } \int_{\mathcal{C}} \phi(s)\frac{N!}{s\left(s-1\right)\cdots\left(s-N\right)}\,ds,\label{2}
\end{equation}
where $\mathcal{C}$ is a positively oriented closed curve that lies in the domain of analyticity of $\phi(s)$, encircles $[n_{0}, N]$, and does not include any of the integers $0,1,\cdots,n_{0}-1$. 
\end{lemma}
\textbf{Proof}. This is a direct application of residue calculus, taking into account contributions of the simple poles at the integers $n_{0}, \cdots,N$. For details, see, \cite{FS}.  $\hfill \blacksquare$
\begin{remark} 
The kernel in relation (\ref{2}) is also expressible in terms of Gamma functions and is known as the Rice kernel.
\end{remark}
\begin{theorem}\label{thm:poles}
Set
\begin{equation}
\phi_{N}(s):=\Bigg(1-\frac{\binom{N-s}{d}}{\binom{N}{d}}\Bigg)^{-1}=\Bigg(1-\prod_{j=0}^{d-1}\left(1-\frac{s}{N-j}\right)\Bigg)^{-1}\label{00}
\end{equation}
and
\begin{equation} 
Q_{N}(s):=1-\prod_{j=0}^{d-1}\left(1-\frac{s}{N-j}\right).\label{3}
\end{equation}
Then, the poles of $\phi_N(s)$ are the zeros of the polynomial $Q_N(s)$. Clearly, $s=0$ is a simple pole of $\phi_{N}(s)$. Moreover, if \(d\) is odd, then the only real pole is \(s=0\). In case where \(d\) is even, then the only real poles are
    \[
    s=0
    \qquad\text{and}\qquad
    s=2N-d+1;
    \]
In particular, all poles of $\phi_N(s)$ are simple. The $d-1$ nonzero zeros of $Q_{N}(s)$ are denoted as $\omega_{1,N},\, \omega_{2,N},\cdots,\omega_{d-1,N}$ (for all positive integer values of $d$; when $d=1$ there are none, and $s=0$ is the only pole).
\end{theorem}
\textbf{Proof.} See, Appendix.\\
 \begin{figure}[h!]
  \centering
  \includegraphics[width=0.82\textwidth]{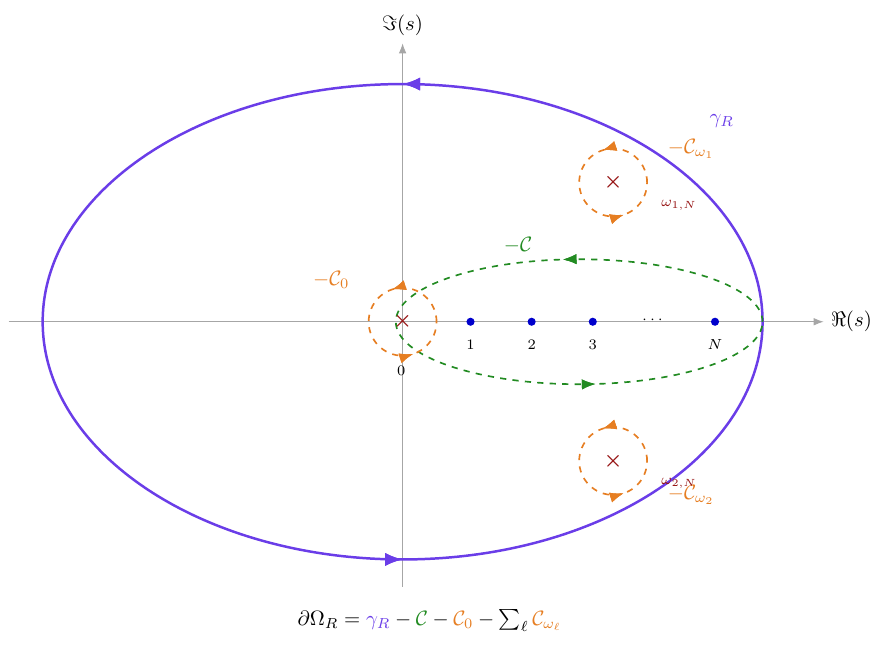}
  \caption{The region $\Omega_{R}$.}\label{fig:contour}
\end{figure} 
By invoking Lemma~\ref{firstlemma} in Theorem~\ref{theorem1} we have
\[
\mathbb {E}[T_{N,d}]
=
\frac{(-1)^{N+1}}{2\pi i}
\int_{\mathcal C}\phi_N(s)K_N(s)\,ds,
\]
where
\begin{equation}
K_{N}(s):=\frac{N!}{s\left(s-1\right)\cdots\left(s-N\right)}. \label{KN}
\end{equation}
Set
\begin{equation}
G_N(s):=(-1)^{N+1}K_N(s)\phi_N(s).
 \label{GN}
\end{equation}
Hence,
\begin{equation}
\mathbb E[T_{N,d}]
=
\frac{1}{2\pi i}\int_{\mathcal C}G_N(s)\,ds.
\label{3bb}
\end{equation}
\begin{theorem}\label{thm:resdecomp}
Regarding the average of the random variable $T_{N,d}$ of the d-CCP problem we have
\begin{equation}
\mathbb {E}[T_{N,d}]
=
-
\operatorname{Res}_{s=0}G_N(s)
-
\sum_{j=1}^{d-1}\operatorname{Res}_{s=\omega_{j,N}}G_N(s).\label{3b}
\end{equation}
\end{theorem}
\textbf{Proof.} We will evaluate the contour integral of (\ref{3bb}) by means of the residue theorem. We denote by \(\gamma_R\) a large positively oriented circle of radius \(R\), by  $C_{0}$ a closed, positively oriented curve of small radius around zero, and, by $C_{\omega_{j}}$ closed, positively oriented curves of small radius around $\omega_{j,N},\, j=1,2,\dots, d-1,$ respectively. Let, also, \(\Omega_R\) denotes the open region enclosed by \(\gamma_R\), minus the closed regions enclosed by the curves $C_{0}$ and $C_{\omega_{j}},\, j=1,2,\dots, d-1$. Hence, for the boundary of \(\Omega_R\) we have (see Figure~\ref{fig:contour})
\begin{equation}
\partial\Omega_R
=
\gamma_R-\mathcal C-\mathcal C_0-\sum_{j=1}^{d-1}\mathcal C_{\omega_j}.
\label{013}
\end{equation}
Clearly, \(G_N\) is analytic inside \(\Omega_R\). Hence, by Cauchy's theorem,
\begin{equation}
\int_{\partial\Omega_R}G_N(s)\,ds=0.
\label{014}
\end{equation}
From relation (\ref{013}) we get
\begin{equation}
\int_{\gamma_R}G_N(s)\,ds
-
\int_{\mathcal C}G_N(s)\,ds
-
\int_{\mathcal C_0}G_N(s)\,ds
-
\sum_{j=1}^{d-1}\int_{\mathcal C_{\omega_{j}}}G_N(s)\,ds
=
0.
\label{015}
\end{equation}
Now, we let \(R\to\infty\). Since
\[
K_N(s)=\mathcal{O}(|s|^{-N-1}),
\qquad
\phi_N(s)=\mathcal{O}(|s|^{-d}),
\qquad |s|\to\infty,
\]
it follows that
\begin{equation}
G_N(s)=\mathcal{O}(|s|^{-N-d-1}).
\label{016}
\end{equation}
Therefore,
\begin{equation}
\left|\int_{\gamma_R}G_N(s)\,ds\right|
\le
2\pi R\,\sup_{|s|=R}|G_N(s)|
=
\mathcal{O}(R^{-N-d})\to 0.
\label{017}
\end{equation}
Passing to the limit in (\ref{015}) we get
\begin{equation}
\int_{\mathcal C}G_N(s)\,ds
=
-
\int_{\mathcal C_0}G_N(s)\,ds
-
\sum_{j=1}^{d-1}\int_{\mathcal C_{\omega_{j}}}G_N(s)\,ds.
\label{018}
\end{equation}
Dividing by \(2\pi i\) and using the definition of the residue we obtain (\ref{3b}).    $\hfill \blacksquare$

\subsection{Asymptotic analysis}
\textit{Step I: The Residue at zero.}\,\,Let us now turn our attention to the polynomial $Q_{N}(s)$ of relation (\ref{3}) around $s=0$ (in the complex plane). As $N\rightarrow \infty$ we have the following
\begin{lemma}\label{lem:res0}
\begin{equation}
\operatorname{Res}_{s=0}G_N(s)
=-\left[
\frac{N}{d}\log N
+\frac{\gamma N}{d}
-\frac{d-1}{2d}\log N
+\left(\frac{1}{2}-\frac{d-1}{2d}\gamma\right)
+\mathcal{O}\left(\frac{\log N}{N}\right)\right].
\label{res0lemma}
\end{equation}
\end{lemma}
\textit{Proof.} For the polynomial $Q_{N}(s)$ of relation (\ref{3}) we have
\begin{equation}
Q_{N}(s)=A(N)s-B(N)s^{2}+\mathcal{O}(s^{3}),\label{4}
\end{equation}
where,
\begin{equation}
A(N):=\sum_{j=0}^{d-1}\frac{1}{N-j} \label{5}
\end{equation}
and
\begin{equation}
B(N):=\sum_{0\leq i<j \leq d-1}\frac{1}{\left(N-i\right)\left(N-j\right)}. \label{6}
\end{equation}
By invoking relation (\ref{00}) we easily get
\begin{equation}
\phi_N(s)=\frac{1}{A(N) s}+\frac{B(N)}{A(N)^2}+\mathcal{O}(|s|),
\qquad s\to 0,\quad s\in\mathbb C. \label{7}
\end{equation}
Since all the poles of $K_{N}(s)$ are simple we have
\begin{align}
(-1)^{N+1} K_{N}(s)&=\frac{(-1)^{N+1} N!}{\prod_{j=0}^{N}\left(s-j\right)} \nonumber \\
&=\sum_{m=0}^{N}\frac{A_{m}}{s-m},\label{8}
\end{align}
where,
\begin{align}
A_{m}:&=\operatorname{Res}_{s=m}\left[(-1)^{N+1}K_{N}(s)\right] \nonumber \\
&= \lim_{s\to m} (s-m)
\frac{(-1)^{N+1}N!}{\prod_{j=0}^{N}(s-j)} \nonumber  \\
&= (-1)^{m+1}\binom{N}{m},
\quad m=0,1,\ldots,N.\label{9}
\end{align}
Since,
\begin{equation*}
\sum_{m=1}^{N}(-1)^{m+1}\binom{N}{m}\frac{1}{m}=H_{N},
\end{equation*}
where, $H_N = \sum_{j=1}^{N}\frac{1}{j}$, is the $N-$th Harmonic number relation (\ref{8}) in view of (\ref{9}) yields
\begin{align}
(-1)^{N+1} K_{N}(s)&=-\frac{1}{s}-H_{N}\left(1+\mathcal{O}(|s|)\right).\label{10}
\end{align}
In view of (\ref{7}) and (\ref{10}), relation (\ref{GN}) yields
\begin{equation}
-\operatorname{Res}_{s=0}G_N(s)=\frac{H_N}{A(N)}+\frac{B(N)}{A(N)^{2}},
\label{11}
\end{equation}
i.e., the coefficient of $-\frac{1}{s}$. From relation (\ref{5}) we have
\begin{align}
\frac{1}{A(N)}&=\frac{N}{d}\left[1+\frac{d-1}{2N}+\frac{\left(d-1\right)\left(2d-1\right)}{6N^{2}}+\mathcal{O}(\frac{1}{N^{3}})\right]^{-1} \nonumber \\
&=\frac{N}{d}-\frac{d-1}{2d}-\frac{d^{2}-1}{12dN}+\mathcal{O}\left(\frac{1}{N^{2}}\right).\label{12}
\end{align}
The full asymptotic expansion of $H_N$ is well known by the celebrated Euler--Maclaurin summation formula (see, e.g., \cite{BO}):
\begin{equation}
H_N = \log N+\gamma+\frac{1}{2N}+\mathcal{O}\left(\frac{1}{N^{2}}\right), \,\,N\rightarrow \infty,\label{13}
\end{equation}
where $\gamma=0.5772156649...$ is the Euler-Mascheroni constant. Finally, it is easy for one to check that
\begin{align}
\frac{B(N)}{A(N)^{2}}
&=
\frac{1}{2}\Bigg(
1-\frac{\sum_{j=0}^{d-1}(N-j)^{-2}}{A(N)^{2}}
\Bigg) \nonumber \\
&=
\frac{1}{2}\Bigg(
1-
\frac{
\frac{d}{N^2}\Big(1+\frac{d-1}{N}
+\mathcal{O}\left(\frac{1}{N^2}\right)\Big)
}{
\frac{d^2}{N^2}\Big(1+\frac{d-1}{N}
+\mathcal{O}\left(\frac{1}{N^2}\right)\Big)
}
\Bigg)\nonumber\\
&=\frac{d-1}{2d}\left(1+\mathcal{O}\left(\frac{1}{N^{2}}\right)\right), \,\,N\rightarrow \infty.
\label{14}
\end{align}
From relations ((\ref{11}), (\ref{12}), (\ref{13}), and, (\ref{14})), we get (as $N\rightarrow \infty$)
\begin{equation}
-\operatorname{Res}_{s=0}G_N(s)
=
\frac{N}{d}\log N
+\frac{\gamma N}{d}
-\frac{d-1}{2d}\log N
+\left(\frac{1}{2}-\frac{d-1}{2d}\gamma\right)
+\mathcal{O}\left(\frac{\log N}{N}\right)
\label{15}
\end{equation}
and the proof is completed.  $\hfill \blacksquare$\\\\
\textit{Step II: The Residue at the non zero poles.} As we will see the contribution (in total) of all the nonzero poles in the residue decomposition of $E[T_{
N,d}]$ is exponentially small.
\begin{lemma}\label{lem:nonzero}
\begin{equation}
\sum_{j=1}^{d-1}\operatorname{Res}_{s=\omega_{j,N}}G_N(s)=\mathcal{O}\left(e^{-A_{d}N}\right),\,\,\, N\rightarrow \infty,\label{16}
\end{equation}
where, $A_{d}$ is a positive constant depending on $d$.
\end{lemma}
\textit{Proof.} By Theorem~\ref{thm:poles}, each \(\omega_{j,N}\) is a simple zero of \(Q_N\), i.e., a simple pole of \(\phi_N\). Hence,
\begin{equation}
\operatorname{Res}_{s=\omega_{j,N}}G_N(s)
=
(-1)^{N+1}K_N(\omega_{j,N})
\operatorname{Res}_{s=\omega_{j,N}}\phi_N(s)
=
(-1)^{N+1}
\frac{K_N(\omega_{j,N})}{Q_N'(\omega_{j,N})}.\label{17}
\end{equation}
Set
\begin{equation}
\mathcal{P}_N(s):=Q_N(Ns)
=
1-\prod_{r=0}^{d-1}\left(1-\frac{s}{1-r/N}\right),
\label{17b}
\end{equation}
and
\[
\mathcal{P}(s):=1-(1-s)^d.
\]
Then the (nonzero) zeros of \(\mathcal{P}\) are the numbers \(a_j\), where 
\begin{equation}
a_j:=1-\zeta_j,
\quad \text{and}\quad
\zeta_j:=e^{2\pi i j/d},
\qquad
j=1,\dots,d-1. \label{15c}
\end{equation}
Of course, all zeros \(a_j\) are simple. For each \(j=1,\dots,d-1\), we consider the closed disk
\[
K_j:=\overline{D(a_j,r_j)}
\]
with radius $r_{j}$ sufficiently small, so that the disks \(K_1,\dots,K_{d-1}\) are pairwise disjoint, \(a_j\) is the unique zero
of \(\mathcal{P}\) in \(K_j\), and, $K_j\cap[0,1]=\varnothing$.
Set
\begin{equation}
K:=\bigcup_{j=1}^{d-1}K_j.\label{K}
\end{equation}
Then, \(K\) is compact and $K\cap[0,1]=\varnothing$.
Of course,
\begin{equation}
 \mathcal{P}_N\to\mathcal{P}\quad \text{locally uniformly on each}\quad  \partial K_j,\label{17c}
 \end{equation}
 and as \(a_j\) is the unique zero of
\(\mathcal{P}\) in \(K_j\) (as already mentioned) we have
\[
\eta_j:=\min_{s\in\partial K_j}|\mathcal{P}(s)|>0.
\]
Hence, for all sufficiently large \(N\),
\[
\sup_{s\in\partial K_j}|\mathcal{P}_N(s)-\mathcal{P}(s)|<\eta_j,
\]
so that
\[
|\mathcal{P}_N(s)-\mathcal{P}(s)|<|\mathcal{P}(s)|,
\qquad s\in\partial K_j.
\]
By Rouch\'e's theorem, \(\mathcal{P}_N\) and \(\mathcal{P}\) have the same
number of zeros in \(K_j\). Since \(\mathcal{P}\) has exactly one zero there,
namely \(a_j\), it follows that \(\mathcal{P}_N\) has exactly one zero
\(s_{j,N}\in K_j\). In particular, $s_{j,N}\to a_j$. Using relation (\ref{17b}) the corresponding (nonzero) zeros of \(Q_N\) are
\begin{equation}
\omega_{j,N}:=Ns_{j,N},
\qquad
j=1,\dots,d-1.  \label{18}
\end{equation}
Again, from (\ref{17b}) we have
\[
Q_N'(s)=\frac{1}{N}\,\mathcal{P}_N'\left(\frac{s}{N}\right).
\]
Since \(\mathcal{P}_N\to \mathcal{P}\) locally uniformly in \(\mathbb{C}\), and each \(\mathcal{P}_N\) is entire,
the derivatives \(\mathcal{P}_N'\) converge to \(\mathcal{P}'\) locally uniformly in \(\mathbb{C}\).
In particular, the convergence is uniform on \(K_j\). Since \(s_{j,N}\in K_j\) and \(s_{j,N}\to a_j\), it follows that
\[
\mathcal{P}_N'(s_{j,N})\to \mathcal{P}'(a_j)\neq 0.
\]
Hence, there exists a constant \(c_d>0\) such that, for all sufficiently large \(N\) and every \(j\)
\[
|\mathcal{P}_N'(s_{j,N})|\ge c_d.
\]
Thus, by relation (\ref{18}) we get
\[
|Q_N'(\omega_{j,N})|
\ge
\frac{c_d}{N}.
\]
By invoking (\ref{17}) in the above we have as $N\rightarrow \infty$
\begin{equation}
\operatorname{Res}_{s=\omega_{j,N}}\phi_N(s)=\mathcal{O}(N).
\label{19}
\end{equation}
Since $s_{j,N}\to a_j$ and \(a_j\) is a simple zero of \(\mathcal{P}\), we have
\[
\mathcal{P}(s)=(s-a_j)h_j(s),
\]
where, \(h_j\) is holomorphic in a neighborhood of \(a_j\), and $h_j(a_j)=\mathcal{P}'(a_j)\neq 0.$ From the continuity of \(h_j\), there exist a  \(r_j>0\) and a constant \(A_j>0\) such that $|h_j(s)|\ge A_j,\,\text{for all }|s-a_j|\le r_j.$ Hence, throughout this disk we have
\[
|\mathcal{P}(s)|\ge A_j|s-a_j|.
\]
For each \(r\in\{0,\dots,d-1\}\), we have
\begin{equation}
1-\frac{s}{1-r/N}
=
1-s-\frac{rs}{N}+\mathcal{O}\left(N^{-2}\right),
\label{20}
\end{equation}
uniformly in \(s\) (on the disk $|s-a_j|\le r_j$). Thus, there exists a
constant \(C_j>0\) such that
\begin{equation}
|\mathcal{P}_N(s)-\mathcal{P}(s)|
\le
\frac{C_j}{N},
\qquad
\text{for all } |s-a_j|\le r_j
\label{21}
\end{equation}
and all sufficiently large \(N\). Now, we choose \(B>0\), such that $A_jB>C_j$, and consider the circle $|s-a_j|=\frac{B}{N}.$ This circle is contained in the fixed disk $|s-a_j|\le r_j$, since $\frac{B}{N}\le r_j$, for sufficiently large $N$. Therefore, both estimates (\ref{20}) and (\ref{21}) are valid on that circle. By the choice of \(B\), we get
\begin{equation}
|\mathcal{P}_N(s)-\mathcal{P}(s)|
<
|\mathcal{P}(s)|,
\qquad
\text{for } |s-a_j|=\frac{B}{N}.
\label{22}
\end{equation}
From Rouch\'e's theorem on the circle $|s-a_j|=\frac{B}{N}$ we see that \(\mathcal{P}_N\) and \(\mathcal{P}\) have the same number
of zeros inside the disk $|s-a_j|<\frac{B}{N}.$ Since, \(\mathcal{P}\) has exactly one zero (there), namely \(a_j\), it follows that \(\mathcal{P}_N\) has also exactly one zero there, i.e., \(s_{j,N}\). Therefore,
\[
s_{j,N}=a_j+O(N^{-1}),
\]
and by invoking relations (\ref{18}) and (\ref{15c}) we have as $N \rightarrow \infty$
\begin{equation}
\omega_{j,N}
=
N\left[\left(1-\zeta_j\right)+\mathcal{O}(\frac{1}{N})\right],\qquad j=1,\cdots,d-1.
\label{23}
\end{equation}
Next we will estimate \(K_N(\omega_{j,N})\). Using  (\ref{KN})  and the relation \(\omega_{j,N}=Ns_{j,N}\), we obtain (after taking logarithms of absolute values)
\begin{equation}
\log|K_N(\omega_{j,N})|
=
\log N!-(N+1)\log N
-\sum_{r=0}^{N}\log\left|s_{j,N}-\frac{r}{N}\right|.
\label{24}
\end{equation}
Stirling's formula yields 
\begin{equation}
\log N!=N\log N-N+\mathcal{O}(\log N),\qquad N\rightarrow \infty.
\label{25}
\end{equation}
For \(s\in K\) (of relation (\ref{K})), let us set
\[
f_s(t):=\log|s-t|,
\qquad t\in[0,1].
\]
Since the compact set $K$ is disjoint from \([0,1]\), we have
\[
\delta:=\operatorname{dist}(K,[0,1])>0.
\]
Therefore, for every \(s\in K\) and every \(t\in[0,1]\), $|s-t|\ge \delta$. Hence, \(f_s\) is \(C^1\) on \([0,1]\), and

\begin{equation*}
|f_s'(t)|
=
\frac{1}{|s-t|}
\le
\frac{1}{\delta},
\qquad
\text{uniformly in } s\in K,\ t\in[0,1].
\end{equation*}

Since \(K\) is compact, there exists \(M>0\) such that
\[
|s-t|
\le
M,
\qquad
s\in K,\ t\in[0,1],
\]
hence,
\begin{equation}
|f_s(t)|
=
|\log|s-t||
\le
\max\{|\log\delta|,|\log M|\},
\qquad
\text{uniformly in } s\in K,\ t\in[0,1].\label {26}
\end{equation}
Using relation (\ref{26}) we get
\begin{equation}
\sum_{r=0}^{N}
\log\left|s-\frac{r}{N}\right|
=
N\int_0^1 \log|s-t|\,dt
+
O(1),
\qquad
\text{uniformly in } s\in K.
\label {27}
\end{equation}
As \(s_{j,N}\in K_j\subset K\), relation (\ref{27}) is valid. Thus, by invoking (\ref{27}) and (\ref{25}), in (\ref{24}) we have
\begin{equation}
\log|K_N(\omega_{j,N})|
=
-N\,\mathcal{I}(s_{j,N})
+
O(\log N),
\label {28}
\end{equation}
where
\begin{equation}
\mathcal{I}(s)
:=
1+\int_0^1 \log|s-t|\,dt.
\label {29}
\end{equation}
We will prove that
\begin{equation}
\mathcal{I}(a_j)>0,
\qquad
j=1,\dots,d-1.
\label {30}
\end{equation}
For \(0<t<1\), and for $a_j=1-\zeta_j$ (see, relation (\ref{15c})) we easily have
\[
|a_j-t|^2
=
t^2+2(1-t)(1-\Re\zeta_j).
\]
Since \(\zeta_j\neq 1\), for \(j=1,\dots,d-1\), we have $\Re\zeta_j<1$. Hence, $|a_j-t|>t, \,\, 0<t<1$. Integrating from \(0\) to \(1\) and invoking relation (\ref{29}) we obtain (\ref{30}). Set
\[
m_d
:=
\min_{1\le j\le d-1}\mathcal{I}(a_j)
\]
Then, clearly, $m_d>0$. Let us set
\[
\alpha_d:=\frac12\,m_d.
\]
As $s_{j,N}\to a_j,\,\,j=1,\dots,d-1$, and \(\mathcal{I}\) is continuous, it follows that for all sufficiently large \(N\),
\begin{equation}
\mathcal{I}(s_{j,N})\ge \alpha_d>0,
\qquad
j=1,\dots,d-1.
\label {31}
\end{equation}
Now Lemma~\ref{lem:nonzero} follows immediately by invoking relations (\ref{31}), (\ref{28}), and (\ref{19}), in relation (\ref{17}). $\hfill \blacksquare$\\
Having Lemmas~\ref{lem:res0}--\ref{lem:nonzero}, Theorem~\ref{thm:resdecomp} yields the following
\begin{theorem}  \label{result}
Consider the  $d$ coupon collector's problem (d--CCP), where in each run we keep \textbf{all} the new observed coupons. Let $T_{N,d}$ be the number of trials until all $N$ different types of coupons, which are uniformly distributed, are collected (at least once). Then, as $N\rightarrow \infty$
\begin{equation}
\mathbb {E}\left[T_{N,d}\right]=
\frac{N}{d}\log N
+\frac{\gamma N}{d}
-\frac{d-1}{2d}\log N
+\left(\frac{1}{2}-\frac{d-1}{2d}\gamma\right)
+\mathcal{O}\left(\frac{\log N}{N}\right).
\label{32}
\end{equation}
\end{theorem}

\subsubsection{Full asymptotic expansion of
\(\mathbb{E}[T_{N,d}]\) as \(N\to\infty\)}

\begin{remark}
As we have seen the asymptotic expansion of \(\mathbb{E}[T_{N,d}]\) is completely determined by the residue at
\(s=0\), since the total contribution of the residues at the nonzero
poles \(\omega_{j,N}\) is exponentially small. Starting from relation
\eqref{11}, we now derive the full asymptotic expansion of
\(\mathbb{E}[T_{N,d}]\) as \(N\to\infty\). In particular, only the
simple pole at \(s=0\) contributes to this expansion.
\end{remark}

Before stating the theorem, we introduce the following notation. Let \(d\ge 2\) be a fixed integer, and for each \(N\ge d\) 
\begin{equation}
H_N:=\sum_{k=1}^{N}\frac1k,
\qquad
A_N:=\sum_{j=0}^{d-1}\frac1{N-j},
\qquad
B_N:=\sum_{0\le i<j\le d-1}\frac1{(N-i)(N-j)}.
\label{33a}
\end{equation}
Let us define
\begin{equation}
S_m(d):=\sum_{j=0}^{d-1}j^m,
\qquad
\sigma_m(d):=\frac{S_m(d)}{d},
\label{33}
\end{equation}
for \(m\ge 0\) and $m\ge 1$ respectively.
We consider the following recurrence relation:
\begin{equation}
c_0(d)=1,
\qquad
c_n(d):=-\sum_{m=1}^{n}\sigma_m(d)c_{n-m}(d),
\qquad n\ge 1,
\label{34}
\end{equation}
as well as, the following
\begin{equation}
e_n(d):=\sum_{\ell=0}^{n}c_\ell(d)c_{n-\ell}(d),
\qquad n\ge 0,
\label{35}
\end{equation}
and,
\begin{equation}
U_m(d):=\sum_{0\le i<j\le d-1}\sum_{\ell=0}^{m}i^\ell j^{m-\ell},
\qquad m\ge 0.
\label{36}
\end{equation}
Finally, for each \(m\ge 0\) set
\begin{equation}
\lambda_m(d):=\frac{c_{m+1}(d)}{d},
\label{37}
\end{equation}
and,
\begin{equation}
\nu_m(d):=
\frac1d\sum_{r=0}^{m+1}c_{m+1-r}(d)h_r
+
\frac1{d^2}\sum_{r=0}^{m}U_r(d)e_{m-r}(d).
\label{38}
\end{equation}
We have the following
\begin{theorem}\label{thm:fullI}
The full asymptotic expansion of the average of the random variable $T_{N,d}$ in the $d$ coupon collector's problem (d--CCP) is given by  
\begin{align}
\mathbb {E}\left[T_{N,d}\right]=&\frac{H_N}{A_N}+\frac{B_N}{A_N^2}\nonumber \\
\sim& \frac{N}{d}\log N+\frac{\gamma}{d}N
+\sum_{m=0}^{\infty}\bigl(\lambda_m(d)\log N+\nu_m(d)\bigr)N^{-m},\quad N\rightarrow \infty,
\label{39}
\end{align}
where the coefficients $\lambda_m(d)$ and $\nu_m(d)$ are defined in relations (\ref{33a})-(\ref{38}).
\end{theorem}
\textbf{Proof.} See, Appendix.\\

\section {Problem II - $\mathbb{E}[T_{N,1}]$}

\begin{theorem}  \label{theorem12}
Consider the  $d$ coupon collector's problem (d--CCP), where in each run we keep \textbf{the least-collected} coupon so far. Let $T_{N,1}$ be the number of trials until all $N$ different types of coupons, which are uniformly distributed, are collected (at least once). Then,
\begin{equation}
\mathbb{E}[T_{N,1}]
=
\sum_{k=0}^{N-1}
\frac{1}{1-\binom{k}{d}/\binom{N}{d}},
\label{80}
\end{equation}
where $\binom{k}{d}=0,$ if $k<d$.
\end{theorem}
\textbf{Proof}. See, \cite{SH}.$\hfill \blacksquare$
\begin{remark}
Again, the result of Theorem~\ref{theorem12} may also be obtained by formulating the problem as a finite Markov chain and applying standard potential theory for hitting times (see, e.g., \cite{N}).
\end{remark}

\subsection{Asymptotics}
We first derive the first four terms of the asymptotic expansion of \(\mathbb{E}[T_{N,1}]\), and then obtain its full asymptotic
expansion as \(N\to\infty\). We have
\begin{align*}
\mathbb{E}[T_{N,1}]
&=\sum_{j=1}^{N}\frac{1}{1-\binom{N-j}{d}/\binom{N}{d}}\\
&=\sum_{j=1}^{N}\Bigg(1-\prod_{r=0}^{d-1}\left(1-\frac{x}{1-r/N}\right)\Bigg)^{-1},
\end{align*}
where we have set $x:=\frac{j}{N}$. Now as $N \rightarrow \infty$
\begin{equation}
\prod_{r=0}^{d-1}
\left(1-\frac{x}{1-r/N}\right)
=
(1-x)^d
-
\frac{d(d-1)}{2N}\,
x(1-x)^{d-1}
+
\mathcal{O}(N^{-2}).
\label{90}
\end{equation}
Thus,
\begin{equation}
\Bigg(
1-\prod_{r=0}^{d-1}\left(1-\frac{x}{1-r/N}\right)
\Bigg)^{-1}
=
\frac{1}{1-(1-x)^d}
-
\frac{d(d-1)}{2N}\,
\frac{x(1-x)^{d-1}}{(1-(1-x)^d)^2}
+
\mathcal{O}(N^{-2}).
\label{91}
\end{equation}
As $x\rightarrow 0$ we have
\[
1-(1-x)^d
=d x+\mathcal{O}(x^2).
\]
Hence,
\begin{equation}
\frac{1}{1-(1-x)^d}
=
\frac{1}{dx}
+
\left(
\frac{1}{1-(1-x)^d}-\frac{1}{dx}
\right),
\label{92}
\end{equation}
as well as,
\begin{equation}
-\frac{d(d-1)}{2}\,
\frac{x(1-x)^{d-1}}{(1-(1-x)^d)^2}
=
-\frac{d-1}{2d}\,\frac{1}{x}
+
q(x),
\label{93}
\end{equation}
where
\begin{equation}
q(x)
:=
-\frac{d(d-1)}{2}\,
\frac{x(1-x)^{d-1}}{(1-(1-x)^d)^2}
+
\frac{d-1}{2d}\,\frac{1}{x}.
\label{94}
\end{equation}
Notice that $q(\cdot)$ extends smoothly to \([0,1]\). By invoking relations (\ref{92})-(\ref{94}) in (\ref{91}) we get
\begin{equation*}
\Bigg(1-\prod_{r=0}^{d-1}\left(1-\frac{x}{1-r/N}\right)
\Bigg)^{-1}
=
\frac{1}{dx}
+
\left(
\frac{1}{1-(1-x)^d}-\frac{1}{dx}
\right)
-
\frac{d-1}{2dN}\,\frac{1}{x}
+
\frac{1}{N}q(x)
+
\mathcal{O}\!\left(\frac{1}{N^2x}\right).
\end{equation*}
Since \(x=j/N\) the above becomes
\begin{equation}
\Bigg(1-\prod_{r=0}^{d-1}\left(1-\frac{x}{1-r/N}\right)
\Bigg)^{-1}
=
\frac{N}{dj}
+
\left(
\frac{1}{1-(1-j/N)^d}-\frac{N}{dj}
\right)
-
\frac{d-1}{2d}\,\frac{1}{j}
+
\frac{1}{N}q(j/N)
+
\mathcal{O}\!\left(\frac{1}{Nj}\right).
\label{95}
\end{equation}
Summing the above over \(j=1,\dots,N\) we get
\begin{equation}
\mathbb{E}[T_{N,1}]
=
\left(
\frac{N}{d}-\frac{d-1}{2d}
\right)
H_N
+
\sum_{j=1}^{N}
\left(
\frac{1}{1-(1-j/N)^d}-\frac{N}{dj}
\right)
+
\frac{1}{N}\sum_{j=1}^{N}q(j/N)
+
\mathcal{O}\!\left(\frac{\log N}{N}\right).
\label{96}
\end{equation}
The asymptotics of the quantity $H_N$ are known as we have already mentioned in relation (\ref{13}) via the  Euler--Maclaurin summation formula. Now, consider the function $u(x):=\frac{1}{1-(1-x)^d}-\frac{1}{dx}$, which also extends smoothly to \([0,1]\). Again, the Euler--Maclaurin summation formula yields as $N \rightarrow \infty$
\begin{equation}
\sum_{j=1}^{N}
\left(
\frac{1}{1-(1-j/N)^d}-\frac{N}{dj}
\right)
=
N\int_0^1
\left(
\frac{1}{1-(1-x)^d}-\frac{1}{dx}
\right)\,dx
+
\frac{d-1}{4d}
+
\mathcal{O}(N^{-1}).
\label{97}
\end{equation}
Similarly,
\begin{equation}
\frac{1}{N}\sum_{j=1}^{N}q(j/N)
=
\int_0^1 q(x)\,dx
+
O(N^{-1}).
\label{98}
\end{equation}
For the integral of the above we apply integration by parts. We get
\begin{equation}
\int_0^1 q(x)\,dx
=
\frac{(d-1)^2}{2d}
-
\frac{d-1}{2}
\int_0^1
\left(
\frac{1}{1-(1-x)^d}-\frac{1}{dx}
\right)\,dx.
\label{99}
\end{equation}
By invoking relations (\ref{13}), (\ref{97}), (\ref{98}), and, (\ref{99}), in relation (\ref{96}) we have the following
\begin{theorem}  \label{result2}
Consider the  $d$ coupon collector's problem (d--CCP), where in each run we keep \textbf{the least} collected coupon so far. Let $T_{N,1}$ be the number of trials until all $N$ different types of coupons, which are uniformly distributed, are collected (at least once). Then, as $N\rightarrow \infty$
\begin{equation}
\mathbb{E}[T_{N,1}]=
\frac{N}{d}\log N
+\left(\frac{\gamma}{d}+
C_d\right)N
-
\frac{d-1}{2d}\log N
+
D_d
+
\mathcal{O}\!\left(\frac{\log N}{N}\right),
\label{102}
\end{equation}
where,
\begin{equation}
C_d:
=
\int_0^1
\left(
\frac{1}{1-(1-x)^d}-\frac{1}{dx}
\right)\,dx,
\label{0100}
\end{equation}
and
\begin{equation}
D_d:
=
\frac{2d^2-3d+3-2\gamma \left(d-1\right)}{4d}
-\frac{d-1}{2}C_d.
\label{101}
\end{equation}
\end{theorem}
\begin{remark}
It is an easy exercise for one to check that the integral $C_d$ of relation ({\ref{0100}}) is always finite, positive, and its value depends only on the positive integer $d>1$. In particular $0<C_d<1$, and 
\begin{equation*}
C_d = -\frac{\psi\left(\frac{1}{d}\right)+\ln d+\gamma}{d},
\end{equation*}
where $\psi(\cdot)$ is the digamma function, i.e., the logarithmic derivative of the gamma function. The exact values of $C_d$, for $d=2,3,4,5,$ and, $6$ are:
\[
\begin{aligned}
C_2
&= \frac{\ln 2}{2}, \\[4pt]
C_3
&= \frac{\pi}{6\sqrt{3}}+\frac{\ln 3}{6}, \\[4pt]
C_4
&= \frac{\pi}{8}+\frac{\ln 2}{4}, \\[4pt]
C_5
&= \frac{1}{5}\left[
\ln 2
+\frac{\pi}{2}\sqrt{\frac{5+2\sqrt{5}}{5}}
+\frac{\sqrt{5}+1}{2}
 \ln\left(\sin\frac{2\pi}{5}\right)
-\frac{\sqrt{5}-1}{2}
 \ln\left(\sin\frac{\pi}{5}\right)
\right], \\[4pt]
C_6
&= \frac{\pi\sqrt{3}}{12}
+\frac{\ln 2}{6}
+\frac{\ln 3}{12}.
\end{aligned}
\]
\end{remark}
 \begin{figure}[h!]
  \centering
  \includegraphics[width=0.72\textwidth]{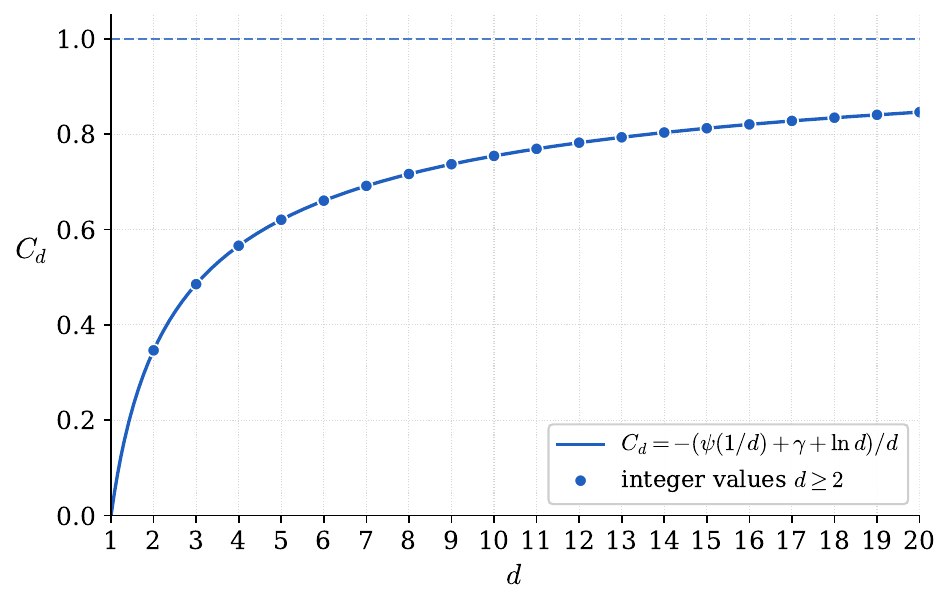}
  \caption{Graph of $C_d$ as $d$ increases.}\label{fig:Cd}
\end{figure} 

\begingroup
\begin{remark}\label{rem:Cdmono}
The constant $C_d$ admits a transparent probabilistic reading. It is the asymptotic
\emph{per-unit-$N$ penalty} paid by the least-collected rule (Problem~II) relative to the
keep-all rule (Problem~I): both means share the term $\frac Nd\log N+\frac\gamma dN$, and
Problem~II carries the additional $NC_d$. From the closed form
$C_d=-(\psi(1/d)+\gamma+\ln d)/d$ one checks that $d\mapsto C_d$ is strictly increasing with
$C_1=0$ and $C_d\uparrow 1$ as $d\to\infty$ (Figure~\ref{fig:Cd}). The monotonicity reflects
that, as the draw size $d$ grows, the least-collected rule increasingly ``wastes'' the
information in the $d$ sampled types by committing to a single one, so the penalty saturates;
the ceiling $C_d<1$ shows that this penalty never exceeds one extra unit of $N$ in the leading
linear order, however large $d$ is.

Moreover, the rate at which $C_d$ approaches its ceiling can be made precise. Using the
expansion of the digamma function about the origin, $\psi(z)=-\frac1z-\gamma+\sum_{k\ge1}(-1)^{k+1}\zeta(k+1)\,z^{k}$
for $|z|<1$, with $z=1/d$, the closed form $C_d=-(\psi(1/d)+\gamma+\ln d)/d$ yields the
asymptotic expansion
\[
C_d=1-\frac{\ln d}{d}-\sum_{k\ge 2}\frac{(-1)^{k}\zeta(k)}{d^{\,k}}
   =1-\frac{\ln d}{d}-\frac{\zeta(2)}{d^{2}}+\frac{\zeta(3)}{d^{3}}-\frac{\zeta(4)}{d^{4}}+\cdots,
\qquad d\to\infty,
\]
where $\zeta(\cdot)$ is the Riemann zeta function and $\zeta(2)=\pi^2/6$. In particular
$1-C_d\sim (\ln d)/d$, so the penalty closes its gap to the ceiling only logarithmically slowly.
\end{remark}
\endgroup

\begingroup
\subsection{Numerical validation of the expansions}
To illustrate the accuracy of the four-term expansions of Theorems~\ref{result}
and~\ref{result2}, Table~\ref{tab:valid} compares them against the exact values from
Theorems~\ref{theorem1} and~\ref{theorem12}, computed in exact rational arithmetic. Already for
moderate $N$ the relative error is small, and it decays at the predicted rate
$\mathcal{O}(\log N/N)$.
\begin{table}[h!]
\centering
\renewcommand{\arraystretch}{1.2}
\begin{tabular}{c|c|rrr|rrr}
 & & \multicolumn{3}{c|}{\textbf{Problem I}\ \ $\mathbb{E}[T_{N,d}]$}
   & \multicolumn{3}{c}{\textbf{Problem II}\ \ $\mathbb{E}[T_{N,1}]$}\\
$d$ & $N$ & exact & 4-term & rel.\ err & exact & 4-term & rel.\ err\\\hline
$2$ & $10$  & $14.1234$  & $14.1791$  & $3.9\times10^{-3}$ & $17.5558$  & $17.5965$  & $2.3\times10^{-3}$\\
$2$ & $50$  & $111.594$  & $111.609$  & $1.3\times10^{-4}$ & $128.877$  & $128.889$  & $9.3\times10^{-5}$\\
$2$ & $100$ & $258.316$  & $258.324$  & $3.2\times10^{-5}$ & $292.926$  & $292.933$  & $2.3\times10^{-5}$\\
$2$ & $500$ & $1696.76$  & $1696.76$  & $1.2\times10^{-6}$ & $1869.99$  & $1870.00$  & $9.5\times10^{-7}$\\\hline
$3$ & $10$  & $9.0462$    & $9.1394$   & $1.0\times10^{-2}$ & $13.9539$  & $14.0080$  & $3.9\times10^{-3}$\\
$3$ & $50$  & $73.7999$  & $73.8242$  & $3.3\times10^{-4}$ & $98.0914$  & $98.1089$  & $1.8\times10^{-4}$\\
$3$ & $100$ & $171.505$  & $171.519$  & $7.9\times10^{-5}$ & $220.063$  & $220.074$  & $4.7\times10^{-5}$\\
$3$ & $500$ & $1130.20$  & $1130.21$  & $3.0\times10^{-6}$ & $1372.92$  & $1372.92$  & $2.0\times10^{-6}$
\end{tabular}
\caption{Exact means versus the four-term asymptotic expansions of Theorems~\ref{result}
and~\ref{result2}, with relative errors.}
\label{tab:valid}
\end{table}
\endgroup

\begin{theorem}\label{thm:fullII}
The full asymptotic expansion of the expectation of the random variable $T_{N,1}$ (of Theorem \ref{result2}) regarding the $d$ coupon collector problem, where we draw $d$ coupons each time, but keep the \textbf{least} collected coupon so far, is (as $N \rightarrow \infty$)
\[
\mathbb{E}[T_{N,1}]
\sim
\frac{N}{d}\log N
+\left(\frac{\gamma}{d}+
C_d\right)N
+
\sum_{j=0}^{\infty}
\bigl(\alpha_{j+1}(d)\log N+u_j(d)\bigr)N^{-j},
\]
where,
\begin{equation*}
C_d:
=
\int_0^1
\left(
\frac{1}{1-(1-x)^d}-\frac{1}{dx}
\right)\,dx,
\end{equation*}
\[
u_j(d)
=
\gamma\,a_{j+1}(d)
+
\int_0^1 h_{j+1}(x;d)\,dx
+
\frac{a_j(d)+h_j(1;d)-h_j(0;d)}{2}
\]
\[
\qquad
-
\sum_{r=1}^{\lfloor (j+1)/2\rfloor}
\frac{B_{2r}}{2r}\,a_{j+1-2r}(d)
+
\sum_{r=1}^{\lfloor (j+1)/2\rfloor}
\frac{B_{2r}}{(2r)!}
\bigl(
h_{j+1-2r}^{(2r-1)}(1;d)-h_{j+1-2r}^{(2r-1)}(0;d)
\bigr),
\]
and, as usual, $\lfloor x\rfloor$ denotes the greatest integer less than or equal to $x$. The quantities $\alpha_j$ and $h_j$ are defined in the Appendix in relations (\ref{a3}) and (\ref{a4}).
\end{theorem}
\textbf{Proof.} See, Appendix.
\begin{remark}\label{rem:d3}
It is not hard to check that the logarithmic terms in the asymptotics of Theorems~\ref{thm:fullI} and~\ref{thm:fullII} are identical, i.e., the coefficients $\lambda_{m}(d)=\alpha_{m+1}(d),\,\,m=0,1,\cdots$. For example, for $d=3$ one has
\[
\mathbb{E}[T_{N,3}]
=
\frac{N}{3}\log N
+
\frac{\gamma}{3}N
-
\frac{1}{3}\log N
+
\left(\frac12-\frac{\gamma}{3}\right)
-
\frac{2}{9}\frac{\log N}{N}
-
\left(\frac{2\gamma}{9}+\frac{7}{36}\right)\frac1N
+
O\!\left(\frac{\log N}{N^2}\right),
\]

\[
\mathbb{E}[T_{N,1}]
=
\frac{N}{3}\log N
+
\left(\frac{\gamma}{3}+C_3\right)N
-
\frac{1}{3}\log N
+
\left(1-\frac{\gamma}{3}-C_3\right)
-
\frac{2}{9}\frac{\log N}{N}
+
\left(\frac{7}{27}-\frac{\log 3}{9}-\frac{2\gamma}{9}\right)\frac1N
+
O\!\left(\frac{\log N}{N^2}\right),
\]
where
\[
C_3
:=
\int_0^1
\left(
\frac{1}{1-(1-x)^3}-\frac{1}{3x}
\right)\,dx
=
\frac{\log 3}{6}
+
\frac{\pi}{6\sqrt{3}}.
\]

\end{remark}

\begingroup
\begin{corollary}\label{cor:d1}
For $d=1$ both problems reduce to the classical coupon collector's problem, and Theorems~\ref{result}, \ref{result2}, \ref{thm:fullI}, and~\ref{thm:fullII} collapse to the well-known expansion of the mean,
\[
\mathbb{E}\left[T_{N,1}\right]\Big|_{d=1}
=\mathbb{E}\left[T_{N,d}\right]\Big|_{d=1}
=N H_N
=N\log N+\gamma N+\frac{1}{2}+\mathcal{O}\!\left(\frac1N\right),
\qquad N\to\infty.
\]
\end{corollary}
\textbf{Proof.} For $d=1$ one has $C_1=\int_0^1\!\big(\tfrac{1}{x}-\tfrac{1}{x}\big)\,dx=0$, while the
constants in Theorems~\ref{result} and~\ref{result2} satisfy
$\frac{d-1}{2d}\big|_{d=1}=0$ and $D_1=\tfrac12$ (directly from~(\ref{101})). Substituting
$d=1$ into either four-term expansion gives $N\log N+\gamma N+\tfrac12$, which is precisely the
Euler--Maclaurin expansion of $NH_N$. $\hfill\blacksquare$
\endgroup

\section{The distribution of the random variables\,\,\,$T_{N,d}$\, and \,$T_{N,1}$}

As we shall see, in both problems the appropriately normalized completion time converges in distribution to a Gumbel random variable. The required normalization, however, is different in the two cases.
\begin{theorem}\label{thm:gumbel}
A collector receives at each run $d$ distinct coupons ($d$ is a positive integer) and keeps \textbf{all the new} observed coupons (\textit{Problem I}), while he chooses \textbf{the least--collected coupon} at each run (\textit{Problem II}). Let $T_{N,d}$ and $T_{N,1}$ be the random variables denoting, respectively, the number of trials for a complete set of $N$ different types of coupons, which are uniformly distributed. Then, as $N\to\infty$,
\begin{equation}
    \frac{
        T_{N,d}
        -
        \frac{N}{d}\log N
    }{
        \frac{N}{d}
    }
  \Longrightarrow
    G, \label{L1}
\end{equation}
and
\begin{equation}
    \frac{
        T_{N,1}
        -
        \frac{N}{d}\log N
        -
        N C_d
    }{
        \frac{N}{d}
    }
    \Longrightarrow
    G,
\label{L2}
\end{equation}
where
\[
    C_d
    =
    \int_0^1
    \left[
        \frac{1}{1-(1-x)^d}
        -
        \frac{1}{dx}
    \right]\,dx,
\]
and $G$ is the standard Gumbel random variable, with distribution function
\[
    \mathbb{P}(G\le x)=\exp\{-e^{-x}\},
    \qquad x\in\mathbb{R}.
\]
\end{theorem}

\begin{figure}[h!]
  \centering
  \includegraphics[width=0.78\textwidth]{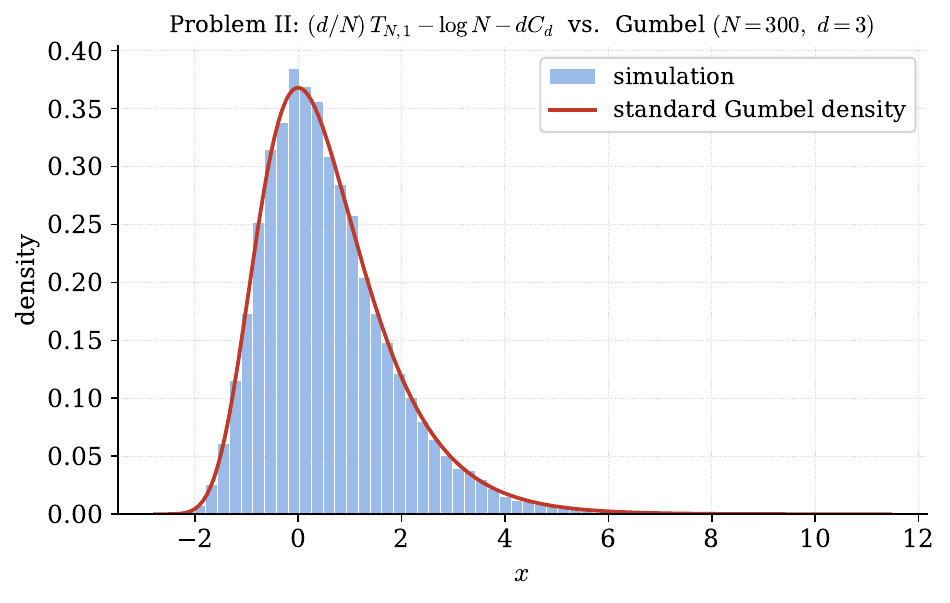}
  \caption{Empirical distribution of the normalised completion time
  $\frac dN T_{N,1}-\log N-dC_d$ in Problem~II (histogram, $40{,}000$ simulations, $N=300$,
  $d=3$) against the standard Gumbel density, illustrating Theorem~\ref{thm:gumbel}.}
  \label{fig:gumbel}
\end{figure}

\textbf{Proof.} (\textit{Problem I}) Let us fix a \(x\in\mathbb{R}\), and set
\[
    t_N(x):
    =
    \left\lfloor
        \frac{N}{d}(\log N+x)
    \right\rfloor .
\]
We will prove that
\[
    \mathbb{P}\left(T_{N,d}\le t_N(x)\right)
    \longrightarrow
    \exp\{-e^{-x}\}.
\]
From relation (\ref{dix}) we have
\[
    \mathbb{P}\left(T_{N,d}\le t_N(x)\right)
    =
    \sum_{r=0}^{N}
    (-1)^r
    \binom{N}{r}
    \left(
        \frac{\binom{N-r}{d}}{\binom{N}{d}}
    \right)^{t_N(x)} .
\]
As $N\rightarrow \infty$ we have $\frac{\binom{N-r}{d}}{\binom{N}{d}}
    =
    1-\frac{rd}{N}
    +
    \mathcal{O}\left(\frac{1}{N^2}\right),$ and for any fixed \(r\)  one also has $\binom{N}{r}=\frac{N^r}{r!}(1+\mathcal{O}(\frac{1}{N})).$
Since,
\[
    t_N(x)
    =
    \frac{N}{d}(\log N+x)+O(1),
\]
and 
\begin{equation*}
\ln\left(1-x\right)=-x-\frac{x^2}{2}+\mathcal{O}\left(x^{3}\right),\,\,x\rightarrow 0^{+},
\end{equation*}
we get
\[
    \mathbb{P}\left(T_{N,d}\le t_N(x)\right)
  =
    \sum_{r=0}^{N}
    (-1)^r
    \frac{e^{-rx}}{r!}\left(1+\mathcal{O}\left(\frac{1}{N}\right)\right).
\]
Taking limits in the above we get the convergence in distribution of relation (\ref{L1}). Indeed, the terms are dominated in absolute value by the summable sequence $e^{-rx}/r!$ uniformly in $N$, so dominated convergence permits passing to the limit term by term, yielding
\[
\lim_{N\to\infty}\mathbb{P}\left(T_{N,d}\le t_N(x)\right)
=\sum_{r=0}^{\infty}\frac{(-1)^r e^{-rx}}{r!}
=\exp\{-e^{-x}\},
\]
which is exactly the standard Gumbel distribution function.
\begin{remark}
For an alternative derivation of relation \eqref{L1}, the interested reader may follow the proof strategy of Theorem 2.1 in \cite{NEAL}.
\end{remark}
(\textit{Problem II})

When \(j\) coupon types are still missing, a new trial fails to produce a new coupon, if all the
\(d\) observed coupon types belong to the already collected set of coupons. Hence, the corresponding success probability is
\begin{equation}
    p_{j,N}
    =
    1-
    \frac{\binom{N-j}{d}}{\binom{N}{d}}.
 \label{pj}
\end{equation}
Therefore, the waiting time \(Y_{j,N}\) spent in phase \(j\), i.e.,
the number of trials needed to go from \(j\) missing coupon types to
\(j-1\) missing coupon types, follows a geometric distribution with parameter $p_{j,N}$. The random variables \(Y_{1,N},\ldots,Y_{N,N}\) are, clearly, independent (since, $ \mathbb{P}(Y_{j,N}=m \mid \text{past})= (1-p_{j,N})^{m-1}p_{j,N},\quad m=1,2,\ldots$). Hence,
\begin{equation}
    T_{N,1}= \sum_{j=1}^{N}Y_{j,N}. \label{tn}
\end{equation}
Let us fix a \(t\in\mathbb{R}\) and consider the following characteristic function 
\[
\mathbb{E}\left[e^{it \widetilde{T}_{N,1}}\right]
  :=  \mathbb{E}
    \left[
        \exp\left\{
            it\left(
                \frac{d}{N}T_{N,1}
                -
                \log N
                -
                dC_d
            \right)
        \right\}
    \right].
\]
Using (\ref{tn}) and independence we have
\[
    e^{-it(\log N+dC_d)}
    \prod_{j=1}^{N}
    \mathbb{E}
    \left[
        e^{it\frac{d}{N}Y_{j,N}}
    \right].
\]
Since $Y_{j,N}$ is a geometric r.v. we have, for each \(j\),
\begin{align*}
    \mathbb{E}
    \left[
        e^{it\frac{d}{N}Y_{j,N}}
    \right]
    &=
    \sum_{m=1}^{\infty}
    e^{it\frac{d}{N}m}
    (1-p_{j,N})^{m-1}p_{j,N}\\
     &=
    \frac{
        p_{j,N}e^{itd/N}
    }{
        1-(1-p_{j,N})e^{itd/N}
    },
\end{align*}
where the last equality follows by summing the geometric series. Therefore, 
\begin{equation}
\mathbb{E}\left[e^{it \widetilde{T}_{N,1}}\right]=
    e^{-it(\log N+dC_d)}
    \prod_{j=1}^{N}
    \frac{
        p_{j,N}e^{itd/N}
    }{
        1-(1-p_{j,N})e^{itd/N}
    }. \label{cf}
\end{equation}
By expanding the exponential: $e^{itd/N}
    =
    1+\frac{itd}{N}
    -
    \frac{t^2d^2}{2N^2}
    +
    \mathcal{O}\left(\frac{1}{N^3}\right),
$
as $N\rightarrow \infty$, we easily get
\begin{equation}
    \frac{
        p_{j,N}e^{itd/N}
    }{
        1-(1-p_{j,N})e^{itd/N}
    }
    =
    \frac{
        \frac{N}{d}p_{j,N}
    }{
        \frac{N}{d}p_{j,N}-it
    }
    \left[
        1+
        \mathcal{O}\left(
            \frac{1}{
                N\left(1+\frac{N}{d}p_{j,N}\right)
            }
        \right)
    \right] \label{cf1}
\end{equation}
uniformly w.r.t. $j=1,2,\cdots,N$. Using relation (\ref{cf1}) in (\ref{cf}) we will estimate the total error of the resulting product. From (\ref{pj}) we, clearly, have $p_{j,N}\geq \frac{j}{N}$. Hence,
\begin{align*}
    \sum_{j=1}^{N}
    \frac{1}{
        N\left(1+\frac{N}{d}p_{j,N}\right)
    }
    =&
   \mathcal{O}\left(\frac{\log N}{N}\right),
\end{align*}
uniformly in $j=1,2,\cdots,N$. Thus,
\begin{equation}
\mathbb{E}\left[e^{it \widetilde{T}_{N,1}}\right]=  e^{-it(\log N+dC_d)}
    \prod_{j=1}^{N}
       \frac{
        \frac{N}{d}p_{j,N}
    }{
        \frac{N}{d}p_{j,N}-it
    }
    \left[
        1+
        \mathcal{O}\left(
            \frac{\log N}{N}\right)    
      \right]. \label{cf2}
\end{equation}
From now on we focus on the quantity
\begin{equation}
P(N):=N^{-it}
\prod_{j=1}^{N}
\frac{
\frac Nd p_{j,N}
}{
\frac Nd p_{j,N}-it
}=
\left[
N^{-it}
\prod_{j=1}^{N}
\frac{j}{j-it}
\right]
\left[
\prod_{j=1}^{N}
\frac{
\frac{
\frac Nd p_{j,N}
}{
\frac Nd p_{j,N}-it
}
}{
\frac{j}{j-it}
}
\right].
\label{cf3}
\end{equation}
Regarding the first bracket of relation (\ref{cf3}) we have
\begin{equation*}
\prod_{j=1}^{N}
\frac{j}{j-it}
=
\Gamma(1-it)
\frac{\Gamma(N+1)}{\Gamma(N+1-it)},
\end{equation*}
where $\Gamma(\cdot)$ is the Gamma function. Since, $
\frac{\Gamma(N+1)}{\Gamma(N+1-it)}
\sim
N^{it}$ as $N\rightarrow \infty$,  we get
\begin{equation}
N^{-it}
\prod_{j=1}^{N}
\frac{j}{j-it}
\longrightarrow
\Gamma(1-it).
\label{cf4}
\end{equation}
Now, regarding the second bracket of relation (\ref{cf3}) we choose a fixed integer \(M\) such that $M>2d|t|$. Then, for every \(j>M\), we have $\left|\frac{it}{j}\right|
<
\frac12$, and since $p_{j,N}\ge \frac{j}{N}$, we have $
\left|
\frac{it}{\frac Nd p_{j,N}}
\right|
\le
\frac{d|t|}{j}
<
\frac12$.
Then,
\begin{equation}
\prod_{j=1}^{N}
\frac{
1-\frac{it}{j}
}{
1-\frac{it}{\frac Nd p_{j,N}}
}
=
\left[
\prod_{j=1}^{M}
\frac{
1-\frac{it}{j}
}{
1-\frac{it}{\frac Nd p_{j,N}}
}
\right]
\left[
\prod_{j=M+1}^{N}
\frac{
1-\frac{it}{j}
}{
1-\frac{it}{\frac Nd p_{j,N}}
}
\right].
\label{cf5}
\end{equation}
For any fixed \(j\) we easily have from  (\ref{pj}) $p_{j,N}=\frac{dj}{N}+\mathcal{O}(N^{-2})$, and therefore $\frac Nd p_{j,N} \sim j$. It follows that, for fixed $M$, the first product in relation (\ref{cf5}) tends to $1$ as $N\to\infty$. Next, we treat the tail product of (\ref{cf5}). Taking logarithms yields
\[
\log
\left[
\prod_{j=M+1}^{N}
\frac{
1-\frac{it}{j}
}{
1-\frac{it}{\frac Nd p_{j,N}}
}
\right]
=
\sum_{j=M+1}^{N}
\left[
\log\left(1-\frac{it}{j}\right)
-
\log\left(
1-\frac{it}{\frac Nd p_{j,N}}
\right)
\right].
\]
From the expansion of the logarithm, i.e., (for \(|z|\le1/2\), $\log(1-z)=-z+\mathcal{O}(z^2)$), and since $j>M>2d|t|$ one has, uniformly in $N$
\begin{equation}
\log
\left[
\prod_{j=M+1}^{N}
\frac{
1-\frac{it}{j}
}{
1-\frac{it}{\frac Nd p_{j,N}}
}
\right]
=
it
\sum_{j=M+1}^{N}
\left[
\frac{1}{\frac Nd p_{j,N}}
-
\frac1j
\right]
+
\mathcal{O}\left(\frac1M\right).
\label{cf6}
\end{equation}
Using relation (\ref{pj}) we have
\[
\frac{1}{\frac Nd p_{j,N}}
-
\frac1j
=
\frac1N
\left[
\frac{d}{
1-\frac{\binom{N-j}{d}}{\binom Nd}
}
-
\frac1{j/N}
\right].
\]
It is easy for one to check that 
\[
\frac{\binom{N-j}{d}}{\binom Nd}
=
\left(1-\frac{j}{N}\right)^d
+
\mathcal{O}\left(\frac{j}{N^2}\right).
\]
Using that $p_{j,N}
\ge
\frac{j}{N}$ we get that uniformly with respect to \(j\):
\[
\frac{d}{
1-\frac{\binom{N-j}{d}}{\binom Nd}
}
=
\frac{d}{
1-\left(1-\frac{j}{N}\right)^d
}
+
O\left(\frac1j\right).
\]
Thus,
\[
\sum_{j=1}^{N}
\left[
\frac{1}{\frac Nd p_{j,N}}
-
\frac1j
\right]
=
\sum_{j=1}^{N}
\frac1N
\left[
\frac{d}{
1-\left(1-\frac{j}{N}\right)^d
}
-
\frac1{j/N}
\right]
+
O\left(\frac{\log N}{N}\right).
\]
The sum of the RHS of the above is a Riemann sum for the function $f(x)=\frac{d}{1-(1-x)^d}-\frac1x$, which has a removable singularity at \(x=0\), and hence, extends continuously to \([0,1]\). It follows that as $N\rightarrow \infty$
\begin{equation}
\sum_{j=1}^{N}
\frac1N
\left[
\frac{d}{
1-\left(1-\frac{j}{N}\right)^d
}
-
\frac1{j/N}
\right]
=
\int_0^1
\left[
\frac{d}{1-(1-x)^d}
-
\frac1x
\right]dx+\mathcal{O}\left(\frac{1}{N}\right).
\label{cf7}
\end{equation}
Notice that the integral in the RHS of relation (\ref{cf7}) is equal to $d$ times the constant $C_d$ of Theorems~\ref{thm:fullII} and~\ref{thm:gumbel}. Since, $\frac Nd p_{j,N}\to j$ we have for fixed \(M\)
\begin{equation}
\sum_{j=1}^{M}
\left[
\frac{1}{\frac Nd p_{j,N}}
-
\frac1j
\right]
=\mathrm{o}(1).
\label{cf8}
\end{equation}
By invoking relation (\ref{cf7}), the definition of the integral $C_d$, and relation (\ref{cf8}), we get that for fixed \(M\)
\[
\log
\left[
\prod_{j=M+1}^{N}
\frac{
1-\frac{it}{j}
}{
1-\frac{it}{\frac Nd p_{j,N}}
}
\right]
=
itdC_d
+
O\left(\frac1M\right)
+
o(1).
\]
First let \(N\to\infty\), and then let \(M\to\infty\). Hence,
\begin{equation}
\prod_{j=M+1}^{N}
\frac{
1-\frac{it}{j}
}{
1-\frac{it}{\frac Nd p_{j,N}}
}
\to
e^{itdC_d}.
\label{cf9}
\end{equation}
From (\ref{cf9}) and the fact that the first product of relation (\ref{cf5}) tends to $1$ as $N\to\infty$ we get
\begin{equation}
\prod_{j=1}^{N}
\frac{
1-\frac{it}{j}
}{
1-\frac{it}{\frac Nd p_{j,N}}
}
\to
e^{itdC_d}.
\label{cf10}
\end{equation}
Finally, from relations (\ref{cf4}), (\ref{cf3}), and (\ref{cf10}), relation (\ref{cf2}) yields
\[
\mathbb{E}\left[e^{it \widetilde{T}_{N,1}}\right]
  :=  \mathbb{E}
    \left[
        \exp\left\{
            it\left(
                \frac{d}{N}T_{N,1}
                -
                \log N
                -
                dC_d
            \right)
        \right\}
    \right]\to
\Gamma(1-it).
\]
Since \(\Gamma(1-it)\) is the characteristic function of a standard Gumbel random variable, the continuity theorem for characteristic functions gives
\[
\frac dN T_{N,1}
-
\log N
-
dC_d
\Longrightarrow
G
\]
which completes the proof. $\hfill \blacksquare$

\begingroup
\begin{remark}[Rate of convergence]\label{rem:rate}
The proof above is quantitative for Problem II. Indeed, writing
$\widetilde T_{N,1}=\frac dN T_{N,1}-\log N-dC_d$, the estimates~(\ref{cf2}) and~(\ref{cf6})
show that, for each fixed $t$,
\[
\Big|\,\mathbb{E}\big[e^{it\widetilde T_{N,1}}\big]-\Gamma(1-it)\,\Big|
=\mathcal{O}\!\left(\frac{(1+|t|)\log N}{N}\right),
\qquad N\to\infty,
\]
the implied constant being absolute. Consequently, by a standard smoothing inequality
(see, e.g., \cite{F}), the Kolmogorov distance between the law of $\widetilde T_{N,1}$ and
the standard Gumbel law $G$ satisfies
\[
\sup_{x\in\mathbb R}\Big|\,\mathbb{P}(\widetilde T_{N,1}\le x)-\mathbb{P}(G\le x)\,\Big|
=\mathcal{O}\!\left(\frac{\log N}{N}\right).
\]
Thus the convergence in~(\ref{L2}) holds at rate $\log N/N$. The same conclusion holds for
Problem~I; the alternating-sum representation in the proof of~(\ref{L1}) gives the
identical order.
\end{remark}
\endgroup

\begingroup
\section{The variance of $T_{N,d}$ and $T_{N,1}$}
Beyond the mean and the limiting law, the second moment is of independent interest and is
the natural next object of study (cf.\ the variance analysis of the classical problem in
\cite{DP}). We record here the leading-order behaviour of the variance in both problems.
We begin with Problem~II, where the phase decomposition makes the computation transparent.

\begin{theorem}\label{thm:varII}
Consider Problem~II, where at each run the collector keeps the least--collected coupon. Then,
as $N\to\infty$,
\begin{equation}
\mathrm{Var}\left[T_{N,1}\right]
=\frac{\pi^2}{6}\,\frac{N^2}{d^2}
-\frac{1}{d^2}\,N\log N
+\mathcal{O}(N).
\label{varII}
\end{equation}
\end{theorem}
\textbf{Proof.} Recall from~(\ref{tn}) that $T_{N,1}=\sum_{j=1}^{N}Y_{j,N}$, where the $Y_{j,N}$
are independent and $Y_{j,N}$ is geometric with parameter $p_{j,N}=1-\binom{N-j}{d}/\binom{N}{d}$
(relation~(\ref{pj})). Hence
\begin{equation}
\mathrm{Var}\left[T_{N,1}\right]
=\sum_{j=1}^{N}\frac{1-p_{j,N}}{p_{j,N}^{2}}
=\sum_{j=1}^{N}\frac{1}{p_{j,N}^{2}}-\mathbb{E}\left[T_{N,1}\right],
\label{varsum}
\end{equation}
since $\sum_j p_{j,N}^{-1}=\mathbb{E}[T_{N,1}]$. By the same Euler--Maclaurin argument used in
Section~3 (now applied to $x\mapsto\big(1-(1-x)^d\big)^{-2}$ after subtracting its
$1/(d^2x^2)$ and $1/x$ singular parts), one obtains
\[
\sum_{j=1}^{N}\frac{1}{p_{j,N}^{2}}
=\frac{\pi^2}{6}\,\frac{N^2}{d^2}
+\frac{d-1}{d^{2}}\,N\log N
+\mathcal{O}(N),
\]
the $N^2$ term coming from the contribution near $j=0$, where $p_{j,N}\sim dj/N$, so that
$p_{j,N}^{-2}\sim N^2/(d^2j^2)$ and $\sum_{j\ge1}j^{-2}=\pi^2/6$. Subtracting
$\mathbb{E}[T_{N,1}]=\frac{N}{d}\log N+\mathcal{O}(N)$ from Theorem~\ref{result2} and using
$\frac{d-1}{d^2}-\frac1d=-\frac1{d^2}$ for the $N\log N$ coefficient yields~(\ref{varII}).
$\hfill\blacksquare$

\begin{remark}\label{rem:K}
The $\mathcal{O}(N)$ term in~(\ref{varII}) is in fact $-K_dN+o(N)$ for an explicit constant
$K_d>0$; for $d=1$ one recovers the classical value $K_1=1+\gamma$, so that
$\mathrm{Var}[T_{N,1}]=\frac{\pi^2}{6}N^2-N\log N-(1+\gamma)N+o(N)$, in agreement with the
literature (see, e.g., \cite{DP}). Numerically, $K_2\approx0.8538$, $K_3\approx0.5915$, and
$K_4\approx0.4547$.
\end{remark}

For Problem~I the variance admits the same leading term. This is consistent with the fact,
established in Section~4, that both normalised completion times converge to the same standard
Gumbel law $G$, whose variance is $\pi^2/6$.

\begin{theorem}\label{thm:varI}
Consider Problem~I, where at each run the collector keeps all the newly observed coupons. Then,
as $N\to\infty$,
\begin{equation}
\mathrm{Var}\left[T_{N,d}\right]
=\frac{\pi^2}{6}\,\frac{N^2}{d^2}\,\bigl(1+o(1)\bigr).
\label{varI}
\end{equation}
\end{theorem}
\textbf{Proof.} From~(\ref{dix}), $\mathbb{P}(T_{N,d}\ge k)=\sum_{m=1}^N(-1)^{m-1}\binom Nm q_m^{k-1}$
with $q_m=\binom{N-m}{d}/\binom Nd$. Using $\mathbb{E}[T_{N,d}^2]=\sum_{k\ge1}(2k-1)\mathbb{P}(T_{N,d}\ge k)$
and summing the resulting geometric series gives the exact closed form
\[
\mathbb{E}\left[T_{N,d}^{2}\right]
=\sum_{m=1}^{N}(-1)^{m-1}\binom{N}{m}
\left(\frac{2}{(1-q_m)^2}-\frac{1}{1-q_m}\right).
\]
Applying the Nørlund--Rice representation of Lemma~\ref{firstlemma} to this difference, exactly
as in Section~2, the dominant contribution again comes from the residue at $s=0$. The double
pole produced by the $(1-q_m)^{-2}$ term yields the leading order $\frac{\pi^2}{6}\frac{N^2}{d^2}$
for $\mathbb{E}[T_{N,d}^2]-\mathbb{E}[T_{N,d}]^2$, while the residues at the nonzero poles remain
exponentially small by Lemma~\ref{lem:nonzero}. This gives~(\ref{varI}). $\hfill\blacksquare$

\begin{remark}
Theorems~\ref{thm:varII} and~\ref{thm:varI} show that, to leading order, the variance is
\emph{insensitive} to the retention rule: both equal $\frac{\pi^2}{6}\frac{N^2}{d^2}$. The
retention rule enters only at order $N\log N$ and below, mirroring the situation for the mean,
where the two problems share every logarithmic coefficient (Remark~\ref{rem:d3}) and differ
only in the linear term $NC_d$.
\end{remark}
\endgroup

\begingroup
\section{An application: coverage in combinatorial motif--based DNA storage}
DNA is an attractive medium for archival data storage because of its density and durability; the principal obstacle to wider adoption is the cost of synthesising DNA nucleotide by nucleotide \cite{SAG}. A promising alternative encodes information not in individual bases but in \emph{combinations} of pre--synthesised oligonucleotides, called \emph{motifs}, drawn from a fixed library of $N$ motifs \cite{PR}. A \emph{combinatorial symbol} is a fixed-size subset of $d$ motifs out of the $N$ available. During writing, the chosen motifs are added to a pool of growing strands, with no control over which motif attaches to which strand; during reading (sequencing), each observation of a symbol reveals the $d$ motifs that constitute it, but in an order and multiplicity that the reader cannot control. To decode a symbol reliably, \emph{all} $d$ of its constituent motifs must be observed at least once across repeated reads. Preuss et al. \cite{PR} first made explicit the connection between this accumulation process and the coupon collector's problem, and this link has since been used to size the required sequencing depth (the so--called \emph{coverage}) in combinatorial DNA--based storage \cite{SOK}.

Our Problem I is exactly this model. Identify the $N$ motifs of a symbol's library with the $N$ coupon types. A single \emph{read} of a symbol corresponds to one run in which $d$ distinct motifs are observed; the reader retains every motif seen for the first time and discards repeats, precisely the ``keep all the new observed coupons'' rule. The waiting time $T_{N,d}$ until every motif has been seen at least once is then the number of reads needed for full recovery of one symbol. Consequently, Theorem~\ref{theorem1} gives the \emph{exact} expected number of reads per symbol,
\[
\mathbb{E}[T_{N,d}]
=\binom{N}{d}\sum_{k=1}^{N}(-1)^{k+1}\binom{N}{k}
\Bigg(\binom{N}{d}-\binom{N-k}{d}\Bigg)^{-1},
\]
while Theorem~\ref{result} gives the practically convenient four--term estimate
\[
\mathbb{E}[T_{N,d}]
=\frac{N}{d}\log N+\frac{\gamma N}{d}
-\frac{d-1}{2d}\log N
+\Big(\tfrac12-\tfrac{d-1}{2d}\gamma\Big)
+\mathcal{O}\!\left(\frac{\log N}{N}\right),
\qquad N\to\infty.
\]
Two features are directly relevant to system design. First, the leading term is $\frac{N}{d}\log N$, so enlarging the number of motifs $d$ packed into each symbol reduces the expected read count by a factor of essentially $d$ relative to the classical single--coupon case ($d=1$). This quantifies, on the reading side, the trade--off noted in \cite{SOK}: a larger combinatorial alphabet raises information density but the per--symbol coverage burden falls only like $1/d$ in the leading order. Second, because $T_{N,d}$ is concentrated around its mean and, suitably normalised, converges to a Gumbel law (Theorem~\ref{thm:gumbel}), the tail $\mathbb{P}(T_{N,d}>t)$ is controlled by the extreme--value distribution; this is what allows a storage system to choose a coverage target that guarantees full recovery with a prescribed probability rather than only in expectation.

As a concrete illustration, consider a library of $N=8$ motifs, the size used in the empirical HelixWorks data set analysed in \cite{SOK}, with $d=4$ motifs per symbol. Theorem~\ref{theorem1} yields the exact value $\mathbb{E}[T_{8,4}]=45949/9867\approx 4.6568$ reads per symbol, against $\mathbb{E}[T_{8,1}]=8H_8=761/35\approx 21.7429$ reads in the classical ($d=1$) regime, a $4.67$--fold reduction. Table~\ref{tab:dna} lists the exact expectations for $d=1,\dots,6$.

\begin{table}[h!]
\centering
\begin{tabular}{c|cccccc}
$d$ & $1$ & $2$ & $3$ & $4$ & $5$ & $6$\\\hline
$\mathbb{E}[T_{8,d}]$ & $21.7429$ & $10.3955$ & $6.5853$ & $4.6568$ & $3.4582$ & $2.6296$
\end{tabular}
\caption{Exact expected number of reads to recover all $N=8$ motifs of a symbol, as the number $d$ of motifs per symbol increases, computed from Theorem~\ref{theorem1}.}
\label{tab:dna}
\end{table}

We stress that Problem~II (keep the least--collected coupon) models a different operational policy, in which a coding or scheduling layer is able to steer each read toward a motif not yet (or least) recovered; the same leading term $\frac{N}{d}\log N$ persists, but the linear coefficient increases by $C_d>0$, reflecting the cost of the more demanding ``balanced coverage'' requirement.

\paragraph{Related work.}
The coverage-depth viewpoint has become an active research theme in coding theory for DNA
storage. Bar-Lev, Sabary, Gabrys, and Yaakobi \cite{BARLEV} introduced the (MDS) coverage-depth
problem and showed that maximum-distance-separable codes minimise the expected number of reads
needed to recover encoded data; Abraham, Gabrys, and Yaakobi \cite{AGY} pursued this further,
and Gruica, Bar-Lev, Ravagnani, and Yaakobi \cite{GRU} analysed the closely related random-access
coverage problem. These works study \emph{coded} systems, where redundancy reduces the coverage
requirement, and the relevant waiting times are typically of order $N$. Our contribution is
complementary and probabilistic: for the \emph{uncoded} combinatorial-symbol channel, where every
one of the $N$ motifs must be observed, we provide the exact mean, its full asymptotic expansion,
the leading-order variance, and the (Gumbel) limiting law of the per-symbol read count, in both
the keep-all and least-collected regimes.
\endgroup

\section{Concluding remarks}
We have studied two $d$ draw versions of the (CCP) and derived full asymptotic expansion for the average of the number of trials for a complete set of $N$ coupons under the different sampling rule, each time. As we have seen, in Remark~\ref{rem:d3}, the logarithmic terms are the same in both cases. Naturally, for $d=1$ the two problems coincide and reduce to the classical (CCP). As expected, Theorems~\ref{thm:fullI} and~\ref{thm:fullII} coincide in the special case where, $d=1$.
As already mentioned the limiting distribution is the same in both cases. Thus, the Gumbel distribution is robust under $d$-coupon sampling. However, the normalization is not the same in the two problems. The difference between them reveals that the final phase of the collection process is sensitive to the specific rule by which the $d$ sampled coupons are used.\\
The same insensitivity is visible at the level of the second moment: by Theorems~\ref{thm:varII} and~\ref{thm:varI}, the variance in both problems is $\frac{\pi^2}{6}\frac{N^2}{d^2}$ to leading order, the retention rule entering only at order $N\log N$. This is exactly the variance of the Gumbel limit rescaled by $N/d$, and it confirms that the $d$-draw mechanism contracts the entire fluctuation scale by the factor $1/d$.\\
Two directions seem worth pursuing. We will leave, for now, as an open question the same problem under non uniform distributions. The obstruction is concrete: our analysis of Problem~I rests on the Nørlund--Rice representation of the alternating sum~(\ref{1}), whose kernel $K_N(s)$ has its poles exactly at the integers; this structure is destroyed once the coupon probabilities $p_j$ are unequal, because the survival probabilities $\big(\sum_{j\in J}p_j\big)$ no longer collapse into a single binomial ratio depending only on $|J|$. Likewise, in Problem~II the phase success probabilities $p_{j,N}$ cease to be a function of the number of missing types alone, so the clean Euler--Maclaurin reduction of Section~3 no longer applies. In some cases the limiting distribution is not Gumbel as we have seen in \cite{DP}, in the so called \textit{linear case}, where $p_{j}=j/\sum_{k=1}^{N}k$. A second natural direction is the multi-set version, in which a full collection requires $m\ge2$ copies of every type; for $d=1$ this is the double Dixie cup problem of Newman and Shepp \cite{NS}, and the interaction of the parameter $m$ with the $d$-draw acceleration appears to be open.

\bigskip
\textbf{Acknowledgements.}\,
The first author thanks Professor Vassilios Gregoriades for helpful comments on the proof of Theorem~\ref{thm:resdecomp}.

\section{Appendix}
\textbf{Proof of Theorem \ref{thm:poles}.} 
The second part of relation (\ref{00}), as well as, the fact that $s=0$ is a simple pole of $\phi_{N}(s)$ are trivial. For the rest of the proof we have\\\\
\medskip \noindent\textbf{Case I: } \(d\) is odd. Assume that  \[ d=2r+1, \qquad c:=N-r. \] 
Then the roots of $Q_{N}(s)$ must satisfy
\[ \prod_{m=-r}^{r}(c+m-s)=\prod_{m=-r}^{r}(c+m). \]
Let us set 
\begin{equation}
u:=c-s. \label{100}
\end{equation}
Then
 \[ \prod_{m=-r}^{r}(u+m)=\prod_{m=-r}^{r}(c+m), \] 
or, equivalently 
\begin{equation}u\prod_{m=1}^{r}(u^2-m^2)=c\prod_{m=1}^{r}(c^2-m^2), \label{1020}
\end{equation}
for all $N\ge d$, which implies $c>r$. It is an easy exercise for one to check that (\ref{1020}) holds, if and only if, $u=c$ (simply, be considering the cases $ |u|<c$ and $u>c$). Hence, from (\ref{100}) we have that the only real pole of $\phi_{N}(s)$ is $s=0$.

 \medskip \noindent\textbf{Case II:} \(d\) even. Assume now, that \[ d=2r, \qquad c:=N-r+\frac12, \qquad a_m:=m-\frac12, \quad m=1,\dots,r. \]
Similarly, the roots of $Q_{N}(s)$ must satisfy  
\[ \prod_{m=1}^{r}(c+a_m-s)(c-a_m-s) = \prod_{m=1}^{r}(c+a_m)(c-a_m). \] 
Setting, again, \[ u:=c-s, \] we get
\begin{equation}
\prod_{m=1}^{r}(u^2-a_m^2)=\prod_{m=1}^{r}(c^2-a_m^2), \label{04}
 \end{equation}
for all \(N\ge d=2r\), which implies \[ c^2-a_m^2>0, \qquad m=1,\dots,r. \]
It is an easy exercise for one to check that (\ref{04}) does not hold if \(|u|<c\), or if \(|u|>c\). It is true, if and only if,
\[
|u|=c,
\]
that is, if and only if, \(s=0\) or \(s=2c=2N-d+1\), which are the only
real zeros of \(Q_N(s)\) when \(d\) is even.\\
To complete the proof we must prove that all zeros of $Q_{N}(s)$ are simple. Assume that $Q_{N}(s_{0})=0$ for some $s_{0}$. Then,
\begin{equation} Q_N'(s_0)=\sum_{j=0}^{d-1}\frac{1}{N-j-s_0}. \label{1004}
 \end{equation} Suppose that \(s_0\notin\mathbb R\), with \[ s_0=x+iy, \qquad y\neq 0. \] Then,
\[\operatorname{Im} \left(\frac{1}{N-j-s_0}\right) = \frac{y}{(N-j-x)^2+y^2}. \] 
Hence,
 \[ Q_N'(s_0)\neq 0. \] If \(s_0=0\), then by (\ref{1004}), \[ Q_N'(0)=\sum_{j=0}^{d-1}\frac{1}{N-j}>0. \] If \(d\) is even and \(s_0=2N-d+1\), then again by (\ref{1004}), \[ Q_N'(2N-d+1) = \sum_{j=0}^{d-1}\frac{1}{N-j-(2N-d+1)} = -\sum_{j=0}^{d-1}\frac{1}{N-d+1+j}<0. \] Thus \(Q_N'(s_0)\neq 0\) for all zeros \(s_0\) of \(Q_N\). $\hfill \blacksquare$

\textbf{Proof of Theorem \ref{thm:fullI}.} For the reader's convenience, we recall that the quantities used in the following proof are defined in relations (\ref{33a})-(\ref{38}). We start with the quantity \(A_N\). For each $j=0,\cdots,d-1$ and for \(M\ge 0\) we have
\[
\frac{1}{N-j}
=
\sum_{m=0}^{M}\frac{j^m}{N^{m+1}}
+
\mathcal{O}(N^{-M-2}),
\qquad N\to\infty.
\]
Summing over $j$ the above yields
\[
A_N\sim \sum_{m=0}^{\infty}\frac{S_m(d)}{N^{m+1}}.
\]
Since, \(S_0(d)=d\) and \(\sigma_m(d)=S_m(d)/d\), for \(m\ge 1\), we have
\[
\sum_{m=0}^{\infty}\frac{S_m(d)}{N^{m+1}}
=
\frac{d}{N}
\left(
1+\sum_{m=1}^{\infty}\sigma_m(d)N^{-m}
\right).
\]
To derive the full asymptotic expansion of \(A_N^{-1}\) we compute the coefficients \(c_n(d)\) such that
\[
\left(
1+\sum_{m=1}^{\infty}\sigma_m(d)x^m
\right)
\left(
\sum_{n=0}^{\infty}c_n(d)x^n
\right)
=1.
\]
Equating coefficients of equal powers of \(x\), we get
\[
c_0(d)=1,
\qquad
c_n(d)
=
-\sum_{m=1}^{n}\sigma_m(d)c_{n-m}(d),
\qquad n\ge 1.
\]
It follows that
\[
\frac{1}{A_N}
\sim
\frac{N}{d}\sum_{n=0}^{\infty}c_n(d)N^{-n},
\]
or, for every positive integer $M$ as $N \rightarrow \infty$
\begin{equation}
\frac{1}{A_N}
=
\frac{N}{d}\sum_{n=0}^{M}c_n(d)N^{-n}
+
\mathcal{O}(N^{-M}).
\label{40}
\end{equation}
Similarly, for the quantity $A_N^{-2}$ we have
\[
\frac{1}{A_N^2}
\sim
\frac{N^2}{d^2}\sum_{n=0}^{\infty}e_n(d)N^{-n},
\]
or, for every positive integer $M$ as $N \rightarrow \infty$
\begin{equation}
\frac{1}{A_N^2}
=
\frac{N^2}{d^2}\sum_{n=0}^{M}e_n(d)N^{-n}
+
\mathcal{O}(N^{1-M}).
\label{41}
\end{equation}
Next, we work on the full asymptotic expansion of \(B_N\). For fixed \(i,j\), we have
\[
\frac{1}{(N-i)(N-j)}
\sim
\sum_{m=0}^{\infty}
\left(
\sum_{\ell=0}^{m}i^\ell j^{m-\ell}
\right)
N^{-m-2}.
\]
Summing over  \(0\le i<j\le d-1\), we get
\[
B_N
\sim
\sum_{m=0}^{\infty}U_m(d)N^{-m-2},
\]
or, equivalently, for every \(M\ge 0\) we have as $N \rightarrow \infty$
\begin{equation}
B_N
=
\sum_{m=0}^{M}U_m(d)N^{-m-2}
+
\mathcal{O}(N^{-M-3}).
\label{42}
\end{equation}
\medskip
Now, regarding the full asymptotic expansion of \(H_N\) everything is known thanks to the celebrated Euler--Maclaurin summation formula (see, e.g., \cite{BO}): 
\[
H_N
\sim
\log N+\sum_{r=0}^{\infty}u_rN^{-r},\,\,\,N\rightarrow \infty,
\]
where 
\[
u_0=\gamma,
\qquad
u_1=\frac12,
\qquad
u_{2q}=-\frac{\mathrm B_{2q}}{2q},\ \ q\ge 1,
\qquad
u_{2q+1}=0,\ \ q\ge 1,
\]
and $B_j$ denotes the $j$-th Bernoulli number  defined by the exponential generating function: $ x/(e^x-1)=\sum_{n=0}^{\infty}B_n \frac{x^n}{n!}.$ Hence, for every \(M\ge 0\),
\begin{equation}
H_N
=
\log N+\sum_{r=0}^{M}u_rN^{-r}
+
\mathcal{O}(N^{-\left(M+1\right)}).
\label{43}
\end{equation}
\noindent
Set
\begin{equation}
\mu_m(d):=
\frac1d\sum_{r=0}^{m+1}c_{m+1-r}(d)h_r.
\label{44}
\end{equation}
Now the full asymptotic expansion of $\frac{H_N}{A_N}$ follows easily from relations (\ref{40}) and (\ref{43}). We have as $N \rightarrow \infty$
\begin{equation}
\frac{H_N}{A_N}
\sim
\frac{N}{d}\log N+\frac{\gamma}{d}N
+\sum_{m=0}^{\infty}\bigl(\lambda_m(d)\log N+\mu_m(d)\bigr)N^{-m}.
\label{45}
\end{equation}
Similarly, regarding the full asymptotic expansion of $\frac{B_N}{A_N^2}$ we have thanks to relations (\ref{41}) and (\ref{42})
\[
\frac{B_N}{A_N^2}
=
\left(
\sum_{r=0}^{M}U_r(d)N^{-r-2}+\mathcal{O}(N^{-M-3})
\right)
\left(
\frac{N^2}{d^2}\sum_{n=0}^{M}e_n(d)N^{-n}+\mathcal{O}(N^{1-M})
\right).
\]
Set
\begin{equation}
\rho_m(d):=
\frac1{d^2}\sum_{r=0}^{m}U_r(d)e_{m-r}(d).
\label{46}
\end{equation}
Thus,
\begin{equation}
\frac{B_N}{A_N^2}
\sim
\sum_{m=0}^{\infty}\rho_m(d)N^{-m}.
\label{47}
\end{equation}
\noindent
For convenience, let
\begin{equation}
\nu_m(d):=\mu_m(d)+\rho_m(d).
\label{48}
\end{equation}
Now by invoking relations (\ref{45}) and (\ref{47}) we have the desired result and the proof is completed. $\hfill \blacksquare$\\

\textbf{Proof of Theorem \ref{thm:fullII}.} As in the proof of Theorem \ref{result2} we have
\begin{equation}
\mathbb{E}[T_{N,1}]=
\sum_{j=1}^{N}\Bigg(
1-\prod_{r=0}^{d-1}\left(1-\frac{x}{1-r/N}\right)
\Bigg)^{-1},
\qquad x=\frac{j}{N}.
\label{a}
\end{equation}
Since $s=0$ is a simple pole of the function $Q_N (s)$ (of relation (\ref{3}), see the proof of Theorem \ref{thm:poles}) we have that the function
\[
S_N(x)
:=
\frac{1-\prod_{r=0}^{d-1}\left(1-\frac{x}{1-r/N}\right)}{x},
\qquad x>0,
\]
is extended smoothly to \(x=0\). For all fixed $d$ and for every fixed \(L\ge 0\) we have as $N\rightarrow \infty$
\[
S_N(x)
=
s_0(x;d)
+
s_1(x;d)N^{-1}
+\cdots+
s_L(x;d)N^{-L}
+
O(N^{-\left(L+1\right)}),
\]
uniformly for \(x\in[0,1]\), where all \(s_m(\,\cdot\,;d)\) are smooth functions, $m=0,1,\cdots,L$. Clearly,
\[
s_0(x;d)
=
\frac{1-(1-x)^d}{x},
\qquad s_0(0;d)=d,\,\,\, x>0,
\]
Of course, $S_N\rightarrow s_0$ uniformly in $[0,1]$, and $s_0$ has a positive minimum value, i.e., $s_0(x;d)\geq  m(d)$. Hence, for sufficiently large $N$: 
$S_N(x)\ge \frac{m(d)}{2}>0, \,x\in[0,1]$. It follows that, for every fixed \(L\ge 0\), there exist smooth
functions \(b_0(\,\cdot\,;d),\dots,b_L(\,\cdot\,;d)\):
\begin{equation}
\frac{1}{S_N(x)}
=
b_0(x;d)
+
b_1(x;d)N^{-1}
+\cdots+
b_L(x;d)N^{-L}
+
O(N^{-\left(L+1\right)})\label{a1}
\end{equation}
uniformly for \(x\in[0,1]\), where 
\begin{equation}
b_0(x;d)=\frac{1}{s_0(x;d)}\quad\text{and}\quad b_m(x;d)=-\frac{1}{s_0(x;d)}\sum_{j=1}^{m}s_j (x;d)\,b_{m-j}(x;d),\,m\geq1.\label{a2}
\end{equation}
Let us set
\begin{equation}
a_m(d):=b_m(0;d).\label{a3}
\end{equation}
It follows that
\begin{equation}
b_m(x;d)=a_m(d)+x\,h_m(x;d),\label{a4}
\end{equation}
where \(h_m(\,\cdot\,;d)\) is smooth on \([0,1]\). By invoking relations ((\ref{a1})--(\ref{a4})) in (\ref{a}), and using that $x=\frac{j}{N}$, we get as $N\rightarrow \infty$
\[
\mathbb{E}[T_{N,1}]=
\sum_{m=0}^{L}
a_m(d)N^{1-m}H_N
+
\sum_{m=0}^{L}
N^{-m}\sum_{j=1}^{N}h_m(j/N;d)
+
O(N^{-L}\log N).
\]
From the Euler-Maclaurin summation formula we have for the harmonic numbers (as already seen) 
\[
H_N
\sim
\log N+\gamma+\frac{1}{2N}
-\sum_{r=1}^{\infty}\frac{B_{2r}}{2r}\,N^{-2r},\, N\rightarrow \infty,
\]
as well as,
\[
\sum_{j=1}^{N}h_m(j/N;d)
\sim
N\int_0^1 h_m(x;d)\,dx
+
\frac{h_m(1;d)-h_m(0;d)}{2}
\]
\[
\qquad
+
\sum_{r=1}^{\infty}
\frac{B_{2r}}{(2r)!}
\bigl(
h_m^{(2r-1)}(1;d)-h_m^{(2r-1)}(0;d)
\bigr)N^{1-2r}.
\]
By collecting terms we get
\[
\mathbb{E}[T_{N,1}]
\sim
\frac{N}{d}\log N
+\left(\frac{\gamma}{d}+
C_d\right)N
+
\sum_{j=0}^{\infty}
\bigl(\alpha_{j+1}(d)\log N+u_j(d)\bigr)N^{-j},
\]
where,
\begin{equation*}
C_d:
=
\int_0^1
\left(
\frac{1}{1-(1-x)^d}-\frac{1}{dx}
\right)\,dx,
\end{equation*}
\[
u_j(d)
=
\gamma\,a_{j+1}(d)
+
\int_0^1 h_{j+1}(x;d)\,dx
+
\frac{a_j(d)+h_j(1;d)-h_j(0;d)}{2}
\]
\[
\qquad
-
\sum_{r=1}^{\lfloor (j+1)/2\rfloor}
\frac{B_{2r}}{2r}\,a_{j+1-2r}(d)
+
\sum_{r=1}^{\lfloor (j+1)/2\rfloor}
\frac{B_{2r}}{(2r)!}
\bigl(
h_{j+1-2r}^{(2r-1)}(1;d)-h_{j+1-2r}^{(2r-1)}(0;d)
\bigr),
\]
and, as usual, $\lfloor x\rfloor$ denotes the greatest integer less than or equal to $x$. This completes the proof. $\hfill \blacksquare$\\



\end{document}